 \documentclass[12pt]{article}
\usepackage{amsmath}
\usepackage{amsthm}
\usepackage{amsfonts}
\usepackage{amssymb}
\usepackage{graphicx}
\usepackage{latexsym}
\usepackage{mathrsfs}
\usepackage{epsfig}

\advance \topmargin by -\headheight
\advance \topmargin by -\headsep
\evensidemargin \oddsidemargin
\newtheorem{theorem}{Theorem}[section]
\newtheorem{thm}{Theorem}[section]

\newtheorem{lem}[theorem]{Lemma}
\newtheorem{proposition}[theorem]{Proposition}
\newtheorem{prop}[theorem]{Proposition}

\newtheorem{question}[theorem]{Question}

\theoremstyle{definition}

\newtheorem{defn}[theorem]{Definition}

\newtheorem{cor}[theorem]{Corollary}

\theoremstyle{remark}
\newtheorem{remark}[theorem]{Remark}

\numberwithin{equation}{section}

\newcommand\EE{\mathbb E}

\newcommand\NN{\mathbb N}
\newcommand\RR{\mathbb R}
\newcommand\ZZ{\mathbb Z}

\newcommand\cB{\mathfrak{B}}

\newcommand\cG{\mathcal{G}}

\newcommand\cI{\mathcal{I}}

\newcommand\cF{\mathfrak{F}}
\newcommand\tcF{\widetilde{\mathfrak{F}}}

\newcommand\cR{\mathcal{R}}

\newcommand\Inn{\operatorname{Inn}}

\newcommand\tpsi{\tilde \psi}

\newcommand\tPsi{\tilde \Psi}

\newcommand\tmu{\widetilde \mu}

\newcommand\tX{\widetilde X}

\newcommand\tZ{\widetilde Z}

\newcommand\abs[1]{\left|#1\right|}

\newcommand\set[1]{\left\{{#1}\right\}}

\def\cc{\curvearrowright}

\def\fU{{\mathfrak{U}}}

\def\sA{{\mathbb{A}}}

\def\sM{{\mathbb{M}}}

\def\II{{\mathbb{I}}}

\def\bK{{\overline{K}}}

\begin{document}
\title{Pointwise ergodic theorems  beyond amenable groups}

\author{Lewis Bowen\footnote{supported in part by NSF grant DMS-0968762, NSF CAREER Award DMS-0954606 and BSF grant 2008274}  ~and Amos Nevo\footnote{supported in part ISF grant  and BSF grant 2008274}}








\begin{abstract}
We prove pointwise and maximal ergodic theorems for probability measure preserving (p.m.p.) actions  of any countable group, provided it admits an essentially free, weakly mixing amenable action of stable type $III_1$. We show that this class contains all irreducible lattices in connected semisimple Lie groups without compact factors. 
We also establish similar results when the stable type is $III_\lambda$, $0 < \lambda < 1$, under a suitable hypothesis.
 
Our approach is based on the following two principles. First, we show that it is possible to generalize the ergodic theory of p.m.p. actions of amenable groups to include p.m.p. amenable equivalence relations. Second, we show that it is possible to reduce the proof of ergodic theorems for p.m.p. actions of a general group to the proof of  ergodic theorems in an associated 
p.m.p.  amenable equivalence relation, provided the group admits an amenable action with the properties stated above. 

%
\end{abstract}
\maketitle

\tableofcontents

\section{Introduction}
\subsection{Background : ergodic theorems for group actions}
Birkhoff's classical pointwise ergodic theorem \cite{Bi31} states the following. If $T: (X,\mu) \to (X,\mu)$ is a p.m.p. (probability measure-preserving) transformation of a standard probability space $(X,\mu)$ then for any $f\in L^1(X,\mu)$, the averages
$$\sA_n[f]:=\frac{1}{n+1}\sum_{i=0}^n f \circ T$$
converge pointwise a.e. to $\EE[f|\cI]$, the conditional expectation of $f$ on the sigma-algebra $\cI$ of $T$-invariant Borel subsets. Convergence in $L^1$-norm had been proven earlier by von Neumann \cite{vN32}. This theorem has been extended in many different directions (see e.g. \cite{Kr85, Te92, As03}). Our focus here is on the possibility of replacing the semi-group $\{T^i\}_{i\ge 0}$ with a general locally compact group (see the survey \cite{Ne05} for further information).

Let $G$ be a locally compact second countable group with a p.m.p.  action on a probability space $(X,\mu)$. Any Borel probability measure $\beta$ on $G$ determines an operator on $L^1(X,\mu)$ defined by
$$\beta(f):= \int f\circ g~d\beta(g),\quad \forall f\in L^1(X,\mu).$$

\begin{defn}
Let $\II$ denote either $\RR_{>0}$ or $\NN$. Suppose  $\{\beta_r\}_{r\in \II}$ is a family of probability measures on $G$.  If for every p.m.p. action $G \cc (X,\mu)$ and every $f\in L^p(X,\mu)$ the functions $ \beta_r(f)$ converge as $ r\to \infty$  pointwise a.e. to the conditional expectation of $f$ on the $\sigma$-algebra of $G$-invariant Borel sets then $\{\beta_r\}_{r\in \II}$ is a {\em pointwise ergodic family in $L^p$}. 
\end{defn}  


Since the time of von-Neumann and Birkhoff, much of the effort in ergodic theory has been devoted to actions of amenable groups. We turn to describe some of the main ergodic theorems established for  them, and then some of those established in the non-amenable case. 

 
\subsubsection{Amenable groups. } A locally compact second countable (lcsc) group $G$ is {\em amenable} if it admits a sequence $\cF=\{\cF_n\}_{n=1}^\infty$ of compact subsets such that for every compact $Q \subset G$, $\lim_{n\to\infty} \frac{\abs{Q\cF_n \Delta \cF_n}}{\abs{\cF_n}}=0$ where $\abs{\cdot}$ denotes left Haar measure. Such a sequence is called {\em F\o lner} or {\em asymptotically invariant}. 

A F\o lner sequence is doubling if it is monotone; namely $\cF_n\subset  \cF_{n+1}$ and satisfies the volume doubling bound, namely there is a constant $C_d>0$,  such that for every $n>0$ 
 $$\Big|  \cF_n^{-1}\cF_n \Big| \le C_d |\cF_n|.$$
 This condition generalizes the doubling condition introduced by Wiener \cite{Wi39} and  Calderon \cite{Ca53}, who 
proved that doubling F\o lner sequences are pointwise ergodic  in $L^1$.

A F\o lner sequence is {\em regular}  if there is a constant $C_{reg}>0$ such that for every $n>0$ 
$$\Big| \bigcup_{i \le n} \cF_i^{-1}\cF_n \Big| \le C_{reg} |\cF_n|.$$

The fact that regular F\o lner sequences are pointwise ergodic sequences in $L^1$  was established by Tempelman \cite{Te72,Te92}, and also by Bewley \cite{Be71}, Chatard \cite{Ch70} and Emerson \cite{Em74}.

A F\o lner sequence is {\em tempered} if there is a constant $C>0$ such that for every $n>0$,
$$\left| \bigcup_{i<n} \cF_i^{-1}\cF_n\right| \le C|\cF_n|.$$

It was shown by E. Lindenstrauss in \cite{Li01} that every F\o lner sequence has a tempered subsequence and every tempered F\o lner sequence is a pointwise ergodic sequence in $L^1$. This is the most general result to date for arbitrary amenable groups. An alternative proof was given by B. Weiss in \cite{We03}. 
The notion of temperedness was introduced and the $L^2$-case was  proven earlier by Shulman \cite{Sh88, Te92}. 

Let us mention that besides the asymptotic invariance inherent in the definition of a F\o lner sequence, there are two other essential ingredients that appear in the proofs of each of the pointwise results stated above.  One is a case-appropriate  generalization of  the Wiener covering argument originally proved for ball averages on Euclidean space, which leads to a weak-type $(1,1)$ maximal inequality for averaging on the sets $\cF_n$ in the group. The other is the Calderon transference principle, which 
reduces the maximal inequality in a general action to the maximal inequality for convolutions on the group itself.
In our discussion below, we will seek to generalize these ingredients beyond the case of 
actions of amenable groups.

\subsubsection{Non-amenable groups.}

The question of a possible generalization of ergodic theorems to arbitrary finitely generated groups was raised already half a century ago by Arnol'd and Krylov. In \cite{AK63} they have generalized  Weyl's equidistribution theorem from dense free groups of rotations of the unit circle to dense free groups of rotations on the unit sphere.  This result motivated the generalization of von-Neumann's mean ergodic theorem from the free group on one generator to the free group on any finite number of generators, established by Y. Guivarc'h \cite{Gu68} using spectral theory. 

Ergodic theorems for measure-preserving actions of arbitrary countable groups were obtained by Oseledets in 1965 \cite{Os65}: he showed that convolution powers of a symmetric probability measure on $\Gamma$ form a pointwise ergodic family.

{\bf Semisimple $S$-algebraic groups}. Techniques based on the spectral theory of unitary representations have been developed and applied to the case where $G$ is a connected semisimple Lie group in \cite{Ne94a, Ne94b, NS94, Ne97, NS97, MNS00}. The more general case of semisimple $S$-algebraic group, and furthermore any lattice subgroup of such a group was established in \cite{GN10}, to which we refer for a more detailed account.  Typically, the averaging sequences studied are the uniform averages over concentric balls (and in some cases,  spheres) centered at the origin.
As an example, we mention that  the free group was handled in \cite{Ne94a, NS94} by viewing it as a lattice in the group of automorphisms of a regular tree, and  in \cite{GN10} as a lattice in $PSL_2(\RR)$. 
 
An important feature of the spectral methods is that the ergodic theorems derived from them 
often exhibit a rate of convergence to the ergodic mean, a phenomenon that cannot arise in the classical amenable context. Thus, when available, spectral methods give results far sharper than 
any other technique,  but their scope is limited to  groups whose unitary representations are well-understood, and to their lattice subgroups.

 {\bf Markov groups}. A most elegant proof of the pointwise ergodic theorem for the free group with respect to spherical averages was given in \cite{Bu00}, using Markov operators techniques developed in \cite{Ro62}. This approach to the ergodic theorem was inspired by earlier related ideas in \cite{Gr99}.
 These techniques extend to a certain extent to groups with a Markov presentation which includes all Gromov-hyperbolic groups. For example, in \cite{BKK11} it is proven that Cesaro averages of spherical averages converge in $L^1$ for every Gromov hyperbolic group with respect to an arbitrary word metric. The identification of the limit function as the ergodic average has recently been obtained in the case of surface groups in \cite{BS10}. 
 
 
 

\subsection{From amenable groups to amenable equivalence relations}

The purpose of this paper is to introduce a general approach for proving pointwise ergodic theorems for countable groups $\Gamma$. This approach has the remarkable feature that it treats amenable and non-amenable groups on an equal footing, and in fact constitutes a direct generalization of the classical techniques of amenable ergodic theory which applies also to non-amenable groups. The two main ideas are as follows. First, we will show that it is possible to reduce the proof of ergodic theorems in measure-preserving $\Gamma$-actions  $(X,\mu)$ to the proof of ergodic theorems in certain associated {\em amenable probability-measure-preserving} equivalence relations.  The amenable equivalence relations are obtained by first choosing an amenable action of $\Gamma$, typically a Poisson boundary $(B,\nu)$, considering its extension $(X\times B,\mu\times \nu)$ by the measure-preserving $\Gamma$-action, and then constructing a probability-measure-preserving amenable sub-relation of the Maharam extension of $X\times B$.  Second, we will show that it is possible  to establish ergodic theorems along F\o lner sets  in p.m.p. amenable equivalence relations, directly generalizing the classical arguments.  Thus when the F\o lner sequence in the equivalence relation is doubling we proceed  by generalizing the arguments of  Wiener and Calderon, or more generally of Tempelman for regular sequences in amenable groups. When the F\o lner sequence  in the equivalence relation is tempered, we proceed by generalizing Weiss' proof of Lindenstrauss' theorem \cite{We03} for tempered sequences in amenable groups. This is accomplished in \S \ref{sec:equivalence}.

\subsection{Statement of one main result}

Next we present one of our main results, with more refined results given later on in the text (specifically, Theorems \ref{random-L-1}, \ref{thm:general} and \ref{thm:general-lambda}). Undefined terminology is explained immediately following the statement of the theorem. 
\begin{thm}\label{thm:show}
Let $\Gamma \cc (B,\nu)$ be a measure class preserving action of a countable group on a standard probability space. We assume the action is essentially free, weakly mixing and stable type $III_1$. Let $\theta$ be the measure on $\RR$ given by $d\theta(t) = e^t dt$. Let $\Gamma \cc (B\times \RR, \nu \times \theta)$ by
$$g(b, t) = \left(gb,t + \log\left( \frac{d\nu\circ g^{-1}}{d\nu}(b) \right)\right).$$
Let $T>0$ be arbitrary,  $I=[0,T]$, and $\theta_I$ be the probability measure on $[0,T]$ given by $d\theta_I(t) = \frac{e^t}{e^T-1} dt$. Let $\cR_I$ be the equivalence relation on $B\times I$ given by restricting the orbit equivalence relation on $B\times \RR$ (so $\cR_I$ consists of all $( (b,t), g(b,t))$ with $g\in \Gamma$ and $(b,t), g(b,t) \in B\times I$). 

Let $\cF=\{\cF_r\}_{r\in \II}$ be a Borel family of subset functions for $(B\times I, \nu \times \theta_I, \cR_I)$. (This implies, in particular that $\cF_r(b, t)$ is a finite subset of the intersection of the $\Gamma$-orbit of $(b,t)$ with $B\times I$). Suppose $\cF$ is either (asymptotically invariant and regular) or (asymptotically invariant, uniform and tempered). Let $\psi \in L^q(B)$ be a probability density function (so $\psi\ge 0$ and $\int \psi ~d\nu = 1$) and define $\zeta^\psi_r:\Gamma \to [0,1]$ by 
$$\zeta^\psi_r(\gamma) := \frac{1}{T} \int_0^T  \int \frac{1}{|\cF_r(b,t)|}1_{\cF_r(b,t)}(\gamma(b,t)) \psi(b)~ d\nu(b)dt.$$
Then $\{\zeta^\psi_r\}_{r\in \II}$ is a pointwise ergodic family in $L^p$ for every $p>1$ with $\frac{1}{p} + \frac{1}{q} \le 1$. If $\psi \in L^\infty$ then $\{\zeta^\psi_r\}_{r\in \II}$ is a pointwise ergodic family in $L \log L$. 
\end{thm}
We also obtain related maximal ergodic theorems under more general hypotheses (see Theorems \ref{thm:maximalg}, \ref{thm:maximalg-discrete}). 

Let us now explain some of the terminology. {\em Essential freeness} of the action $\Gamma \cc (B,\nu)$ means that for a.e. $b \in B$ the stability group $\{g\in \Gamma:~gx=x\}$ is trivial. By an {\em amenable action} we mean an action amenable in the sense of Zimmer \cite{Zi78}. Alternatively, by \cite{CFW81} this is equivalent to the existence of a Borel transformation $S:B \to B$ such that for a.e. $b\in B$, $\Gamma b = \{S^i b:~i\in \ZZ\}$. {\em Weakly mixing} means that if $\Gamma \cc (X,\mu)$ is any ergodic p.m.p. (probability-measure-preserving) action then the product action $\Gamma \cc (B\times X,\nu\times \mu)$ is ergodic. 

The action $\Gamma \cc (B,\nu)$ has {\em type $III_1$} if for every $r, \epsilon>0$ and every positive measure Borel set $A \subset B$, there exists a positive measure Borel set $A_0 \subset A$ and an element $g\in \Gamma \setminus\{e\}$ such that $gA_0 \subset A$ and 
$$\left| \frac{d\nu \circ g}{d\nu}(b) -r \right| < \epsilon$$
for every $b\in A_0$. The action $\Gamma \cc (B,\nu)$ has {\em stable type $III_1$} if for every p.m.p. ergodic action $\Gamma \cc (X,\mu)$ the product action $\Gamma \cc (B\times X,\nu\times \mu)$ has type $III_1$. It should be noted that we also obtain results analogous to Theorem \ref{thm:show} for certain actions of stable type $III_\tau$ for $\tau \in (0,1)$, a notion defined in \S \ref{sec:AR}.

A Borel family $\cF=\{\cF_r\}_{r\in \II}$ of subset functions for $(B\times I, \nu\times \theta_I,\cR_I)$ satisfies
\begin{enumerate}
\item for each $(b,t) \in B\times I$, $\cF_r(b,t)$ is a finite subset of $B\times I$ contained in the $\Gamma$-orbit of $(b,t)$;
\item the set $\{ (b,t,b',t',r) \in B\times I \times B \times I \times \II:~ (b',t') \in \cF_r(b,t) \}$ is Borel.
\end{enumerate}

Let $\cR_I$ denote the equivalence relation on $B\times I$ given by $((b,t), (b',t')) \in \cR_I$ if and only if $(b',t')$ is in the $\Gamma$-orbit of $(b,t)$. Let $\Inn(\cR_I)$ denote the full-group of $\cR_I$: this is the group of all Borel automorphisms of $B\times I$ with graph contained in $\cR_I$. A set $\Psi \subset \Inn(\cR_I)$ {\em generates $\cR$} with respect to $\nu \times \theta_I$ if for $\nu\times \theta_I$-a.e. $(b,t)$ and every $(b',t')$ with $((b,t),(b',t')) \in \cR_I$ there exists $\psi$ in the group generated by $\Psi$  such that $\psi(b,t)=(b',t')$. 

The family  $\cF$ is {\em asymptotically invariant} if $|\cF_r(b,t)| \ge 1$ for a.e. $(b,t)\in B\times I$ and $r\in \II$ and there exists a countable set $\Psi \subset \Inn(\cR)$ which generates $\cR$ such that for every $\psi \in \Psi$ and $\nu\times \theta_I$-a.e. $(b,t)\in B \times I$
$$\lim_{r\to\infty} \frac{|\cF_r(b,t) \Delta \psi(\cF_r(b,t))|}{|\cF_r(b,t)|} =0.$$

The  family $\cF$ is {\em regular} if there is a constant $C_{reg}>0$, also called the regularity constant, such that  for $\nu\times \theta_I$-a.e. $(b,t)\in B\times I$ and every $r>0$ 
$$\Big| \bigcup_{t \le r} \cF_t^{-1}\cF_r(b,t) \Big| \le C_{reg} |\cF_r(b,t)|$$
where $\cF_t^{-1}\cF_r(b,t)$ is the set of all $(b',t')$ such that $\cF_t(b',t') \cap \cF_r(b,t) \ne \emptyset$. The concepts uniform and tempered and described in \S \ref{sec:mer-definitions}. 

\subsection{About the hypotheses}

Theorem \ref{thm:show} and its refinements  (Theorems \ref{random-L-1}, \ref{thm:general} and \ref{thm:general-lambda}), each require the existence of a measure-class preserving action $\Gamma \cc (B,\nu)$ which is essentially free, weakly mixing, amenable and either stable type $III_1$ or type $III_\lambda$ and stable type $III_\tau$ for some $\lambda, \tau \in (0,1)$. So it is natural to ask, when does such an action exist and how can we find one?

First we note that the requirement that the action be essentially free is not very restrictive in the sense that if there is an action satisfying the other conditions then there is an essentially free action which satisfies all the conditions. To explain, let $\Gamma \cc (X,\mu)$ be any essentially free, weakly mixing, probability-measure-preserving action (for example, Bernoulli actions satisfy this property). If $\Gamma \cc (B,\nu)$ is weakly mixing, amenable, type $III_\lambda$ and stable type $III_\tau$ (for some $\lambda,\tau \in [0,1]$) then the product action $\Gamma \cc (X\times B,\mu\times \nu)$ is also  weakly mixing, amenable, type $III_\lambda$ and stable type $III_\tau$. Moreover the product action is essentially free.

Now given a symmetric probability measure $\mu$ on $\Gamma$, we may consider the Poisson boundary $(B,\nu)$ for the random walk with $\mu$-distributed increments. There is a natural action of $\Gamma$ on $(B,\nu)$ which is always amenable \cite{Zi78} and weakly mixing \cite{AL05}. It is not known whether this action is always type $III_\lambda$ for some $\lambda \in (0,1]$ or whether by choosing $\mu$ appropriately one can always require the action to be type $III_\lambda$ for some $\lambda \in (0,1]$. However, under extra hypotheses on $\Gamma$, we do have some answers.  We establish in \S \ref{type III} that all irreducible lattices in connected semisimple Lie groups without compact factors have the property that their action on the Poisson boundary $B=G/P$ has stable type $III_1$. It was shown in \cite{INO08} that Poisson boundaries of Gromov hyperbolic groups are never type $III_0$. 

In future work \cite{BN2} we intend to show that if $\Gamma$ is Gromov hyperbolic then the action on its boundary with respect to the Patterson-Sullivan measure is weakly mixing, amenable, type $III_\lambda$ and stable type $III_\tau$  for some $\lambda, \tau \in (0,1]$. 
We will use this in \cite{BN2} to obtain pointwise ergodic sequences $\{\zeta_r\}_{r=1}^\infty$ for general Gromov hyperbolic groups $\Gamma$, such that each $\zeta_r$ is supported in a spherical shell of constant width.

We conjecture that any countable group admits an action satisfying all the requirements above. If true, then the results of this paper apply to all countable groups.

\subsection{Organization of the paper}

In \S \ref{sec:equivalence} we prove ergodic theorems for amenable equivalence relations. In \S \ref{sec:random} we prove maximal ergodic theorems. In \S \ref{sec:general} we use results of the previous two sections to prove pointwise ergodic theorems when $\Gamma \cc (B,\nu)$ has stable type $III_1$. In \S \ref{sec:lambda} we prove pointwise ergodic theorems when $\Gamma \cc (B,\nu)$ has type $III_\lambda$ and stable type $III_\tau$ for $\lambda,\tau \in (0,1)$.

\subsection*{Acknowledgement} The authors would like to thank the (anonymous) referee for many useful comments, which resulted in significant improvements to the presentation.

\section{Ergodic theorems for amenable equivalence relations}\label{sec:equivalence} 

\subsection{Definition of F\o lner sets and their properties}\label{sec:mer-definitions}

A {\em measured equivalence relation} is a quadruple $(X,\mathcal{B},\mu,\cR)$ where $(X,\mathcal{B},\mu)$ is a standard $\sigma$-finite measure space and $\cR \subset X \times X$ is a Borel equivalence relation. It is {\em discrete} if every equivalence class, denoted $[x]$, is at most countable. It is a {\em probability-measured equivalence relation} if $\mu(X)=1$. To reduce notation, we will usually omit the sigma-algebra from the notation and say that $(X,\mu,\cR)$ is a measured equivalence relation. 

Let $c$ denote counting measure on $X$ (so $c(E) = \# E ~\forall E \subset X$). The measure $\mu$ on $X$ is {\em $\cR$-invariant} if $\mu\times c$ restricted to $\cR$ equals $c\times \mu$ restricted to $\cR$. In this case we say $(X,\mu,\cR)$ is a {\em measure-preserving equivalence relation}. A Borel map $\psi: X \to X$ is an {\em inner automorphism} of $\cR$ if it is invertible with Borel inverse and its graph is contained in $\cR$. Let $\Inn(\cR)$ denote the group of inner automorphisms. This group is frequently called the {\em full group} and denoted by $[\cR]$. If $\mu$ is $\cR$-invariant then $\psi_*\mu=\mu$ for every $\psi \in \Inn(\cR)$. For the rest of this section, we assume $(X,\mu,\cR)$ is a discrete probability measure-preserving (p.m.p.) equivalence relation. 

A {\em subset function} (for $\cR$) is a map $\fU$ on $X$ such that $\fU(x) \subset [x]$ for all $x\in X$. The inverse of $\fU$ is the subset function $\fU^{-1}(y):=\{x \in X:~ y \in \fU(x)\}$. If $\fU_1, \fU_2$ are two subset functions then their product $\fU_1\fU_2$ is the subset function defined by
$$\fU_1\fU_2(x):=\bigcup\{ \fU_1(y):~ y \in \fU_2(x)\}.$$
Their difference $\fU_1\setminus \fU_2$ is defined by $\fU_1\setminus\fU_2(x):=\fU_1(x)\setminus\fU_2(x)$. We write $\fU_1 \subset \fU_2$ if $\fU_1(x) \subset \fU_2(x)$ for a.e. $x$. If $\{\fU_i\}_{i\in I}$ is a family of subset functions their union $\cup_{i\in I} \fU_i$ is the subset function defined by
$$(\cup_{i\in I}\fU_i)(x):= \bigcup_{i\in I}\fU_i(x).$$

A {\em Borel family of subset functions} $\cF=\{\cF_r\}_{r \in \II}$ (for $\cR$) is a family of subset functions $\cF_r$ indexed by a set $\II \in \{\NN, \RR_{>0}\}$ such that $\{(x,y,r) \in X\times X \times \II:~ y \in \cF_r(x)\}$ is a Borel subset of $\cR \times \II$. 
 As noted already, we will always assume that $\cup_{s\le r}\cF_s(x)\subset [x]$ is finite for every $x\in X$ and $r \in \II$. As a result we also have that 
$\bigcup_{t \le r} \cF_t^{-1}\cF_r(x) $ is finite for every $r \in \II$ and every $x\in X$.

Let $\cF$ be a Borel family of subset functions. The definitions below generalize classical concepts. 
\begin{enumerate}


\item A set $\Psi \subset \Inn(\cR)$ {\em generates $\cR$} with respect to $\mu$ if for $\mu\times c$ a.e. $(x_1,x_2) \in \cR$ there exists $\psi \in \langle \Psi \rangle$ such that $\psi(x_1)=x_2$ (where $\langle \Psi \rangle$ denotes the subgroup of $\Inn(\cR)$ generated by $\Psi$). Equivalently $\cR=\bigcup_{\psi\in \langle\Psi\rangle} \text{Graph }(\psi)$, up to a set of $\mu\times c$ measure zero.

\item $\cF$ is {\em asymptotically invariant} if $|\cF_r(x)| \ge 1$ for a.e. $x\in X$ and $r\in \II$ and there exists a countable set $\Psi \subset \Inn(\cR)$ which generates $\cR$ such that for every $\psi \in \Psi$ and $\mu$-a.e. $x \in X$
$$\lim_{r\to\infty} \frac{|\cF_r(x) \Delta \psi(\cF_r(x))|}{|\cF_r(x)|} =0.$$
We say that $\cF$ is {\em F\o lner} if it is asymptotically invariant.





\item $\cF$ is {\em uniform} if  there are constants $C_u, a_r,b_r >0$ (for $r \in \II$) such that
\begin{enumerate}
\item $b_r \le C_u a_r$ for every $r\in \II$, 
\item $a_r \le |\cF_r(x)| \le b_r$ for a.e. $x\in X$,
\item $a_r \le |\cF_r^{-1}(x)| \le b_r$ for a.e. $x\in X$. 
\end{enumerate}
The constant $C_u$ is called the {\em uniformity} constant.

\item $\cF$ is {\em doubling} if (i) $\cF$ is a monotone family namely for $s < r$ we have $\cF_s(x)\subset \cF_{r}(x)$ a.e., and (ii) $\cF$ satisfies the volume doubling bound, namely there is a constant $C_d>0$, called the doubling constant, such that  for $\mu$-a.e. $x\in X$ and every $r>0$ 
$$\Big| \cF_r(x)^{-1}\cF_r(x) \Big| \le C_d |\cF_{r}(x)|.$$

\item $\cF$ is {\em regular} if there is a constant $C_{reg}>0$, also called the regularity constant, such that  for $\mu$-a.e. $x\in X$ and every $r>0$ 
$$\Big| \bigcup_{t \le r} \cF_t^{-1}\cF_r(x) \Big| \le C_{reg} |\cF_r(x)|.$$


\item $\cF$ is {\em tempered} if the index set $\II = \NN$ and there is a constant $C_t$ such that for $\mu$-a.e. $x\in X$ and every $n>0$ 
$$\Big| \bigcup_{m \le n-1} \cF_m^{-1}\cF_n(x) \Big| \le C_t |\cF_n(x)|.$$
$C_t$ is called the {\em tempered} constant.
\end{enumerate}
For a function $f$ on $X$, consider the averages $\sA[f|\cF_r]$ defined by
$$\sA[f|\cF_r](x):=\frac{1}{|\cF_r(x)|} \sum_{x'\in\cF_r(x)} f(x'),\quad \forall x\in X.$$
A subset $E \subset X$ is {\em $\cR$-invariant}  (or $\cR$-saturated) if $(E \times X) \cap \cR = (X \times E) \cap \cR = (E \times E) \cap \cR$ (up to $\mu\times c$-measure zero).  For a Borel function $f\in L^1(X)$, let $\EE[f|\cI(\cR)]$ denote the conditional expectation of $f$ with respect to the $\sigma$-algebra $\cI(\cR)$ of $\cR$-invariant measurable sets.

\subsection{Statement of ergodic theorems for  equivalence relations}
Keeping the notation introduced in \S 2.1,  in \S 2.3 - \S 2.6 we prove the following results : 
\begin{thm}\label{thm:pointwise}
If $\cF$ is either (asymptotically invariant and regular) or (asymptotically invariant, uniform and tempered) then $\cF$ is a (restricted) pointwise ergodic family in $L^1$.  Namely, for every $f\in L^1(X,\mu)$, $\sA[f|\cF_r]$ converges pointwise a.e. to $\EE[f|\cI(\cR)]$ as $r\to\infty$.
\end{thm}
Theorem \ref{thm:pointwise2} below extends Theorem \ref{thm:pointwise} and shows that in fact $\cF$ is an (unrestricted) pointwise ergodic family in $L^1$, namely that the result  passes to class-bijective ergodic extensions. This will be established  in \S 2.6. 

Our method of proof of Theorem \ref{thm:pointwise} follows the classical pattern and is based on the next two theorems.

\begin{thm}[Dense subset of good functions]\label{thm:dense}
If $\cF$ is asymptotically invariant then there exists a dense subset $\cG \subset L^1(X)$ such that for all $f \in \cG$, $\sA[f|\cF_r]$ converges pointwise a.e. to $\EE[f|\cI(\cR)]$ as $r\to\infty$. Moreover, if $L^1_0(X)$ is the set of all functions $f\in L^1(X)$ with $\EE[f|\cI(\cR)]=0$ a.e. then there exists a dense subset $\cG_0 \subset L^1_0(X)$ such that for all $f\in \cG_0$,  $\sA[f|\cF_r]$ converges pointwise a.e. to $0$ as $r\to\infty$.
\end{thm}

For $f\in L^1(X)$, let $\sM[f |\cF] = \sup_{r\in \II} \sA[ |f| |\cF_r]$ where $|f|$ denotes the absolute value of $f$. $\sM[\cdot | \cF]$ is the {\em maximal operator} associated to the family  $\{\sA[\cdot |\cF_r]\}_{r\in \II}$. As we shall see below in the proof of the maximal inequality, the maximal function $\sM[\cdot | \cF]$ is Borel measurable, even when the index set $ \II=\RR$. We can now state


\begin{thm}[Weak $(1,1)$-type maximal inequality]\label{thm:maximal} 
Suppose that $\cF$ is either regular or (asymptotically invariant, uniform and tempered). Then there exists a constant $C>0$ such that for any $f \in L^1(X)$ and any $\lambda>0$,
$$\mu\left(\left\{ x\in X:~\sM[f|\cF]>\lambda \right\}\right) \le \frac{C||f||_1}{\lambda}.$$
In fact $C$ can be taken to be $8C_u^4(1 + C_tC_u)$ in the tempered case and $C_{reg}$ in the regular case.
\end{thm}
\noindent Theorem \ref{thm:maximal2} below extends Theorem \ref{thm:maximal} to arbitrary class-bijective extensions of $(X,\mu,\cR)$. 

We recall that the concept of an amenable group action has been defined in great generality in \cite{Zi78}, and several characterizations of this property have been established, including in  \cite{CFW81}. We will not elaborate here on these results, since they are not directly relevant to our discussion. Instead, we will just note that it has been shown in \cite{Zi78} that the Poisson boundary associated with a non-degenerate random walk on $\Gamma$ is an amenable ergodic action of $\Gamma$ (but it still an open problem whether the type can be $III_0$). Furthermore, in \cite{CFW81} is was shown that amenability of a general equivalence relation is equivalent to the existence of an asymptotically invariant sequence of subset functions (defined taking the Radon-Nikodym derivative into account). Focusing on the case of p.m.p. equivalence relations, which is our main concern, let us  formulate a general existence result for doubling F\o lner sequences, based on another important fact from \cite{CFW81}, namely that amenability of an equivalence  relation is equivalent to the relation being the orbit relation of a $\ZZ$-action, as follows.
\begin{prop}\label{prop:tempered}
Let $(X,\cB,\mu,\cR)$ be an amenable discrete p.m.p. equivalence relation. Then there exists a sequence $\cF=\{\cF_n\}_{n=1}^\infty$ of subset functions which is asymptotically invariant, uniform and doubling.
\end{prop}

\begin{proof}
Without loss of generality, we may assume $(X,\cB,\mu,\cR)$ is ergodic. If $X$ is finite then we may choose $\cF_n(x)=X$ for every $n, x$. So let us assume $X$ is infinite. According to \cite{CFW81, Dy59, Dy63}, there exists a measure-preserving Borel transformation $T:X \to X$ so that $\cR= \{ (x,T^ix):~x\in X, i \in \ZZ\}$ (up to a $\mu\times c$-measure zero subset). Then we may let $\cF_n(x):=\{ T^i x:~ |i| \le n\}$. It is easy to check that $\cF=\{\cF_n\}_{n=1}^\infty$ is asymptotically invariant,  uniform and doubling.



\end{proof}

\subsection{Dense set of good functions}
In this subsection, we prove Theorem \ref{thm:dense}. So assume $\cF$ is asymptotically invariant. Let $\Psi \subset \Inn(\cR)$ be a countable set generating the relation $\cR$ that witnesses the asymptotic invariance. 
\begin{lem}\label{lem:group}
Let $\psi$ be in the subgroup of $\Inn(\cR)$ generated by $\Psi$. Then
$$\lim_{r\to\infty}  \frac{|\cF_r(x) \Delta \psi(\cF_r(x))|}{|\cF_r(x)|} = 0.$$
\end{lem}

\begin{proof}
Observe that if $\psi_1,\psi_2 \in \Psi$ then $|\cF_r(x)\Delta \psi_i^{-1}(\cF_r(x))| = |\psi_i(\cF_r(x))\Delta \cF_r(x)|$. So
$$\lim_{r\to\infty}  \frac{|\cF_r(x) \Delta \psi_i^{-1}(\cF_r(x))|}{|\cF_r(x)|} = 0\,\,,\,i=1,2.$$
Also
\begin{eqnarray*}
|\cF_r(x) \Delta \psi_1\psi_2(\cF_r(x))| &\le& |\cF_r(x) \Delta \psi_1(\cF_r(x))| + |\psi_1(\cF_r(x)) \Delta \psi_1\psi_2(\cF_r(x))|\\
&=&  |\cF_r(x) \Delta \psi_1(\cF_r(x))| + |\cF_r(x) \Delta \psi_2(\cF_r(x))|.
\end{eqnarray*}
Therefore,
$$\lim_{r\to\infty}  \frac{|\cF_r(x) \Delta \psi_1\psi_2(\cF_r(x))|}{|\cF_r(x)|} = 0.$$
Since $\psi_1,\psi_2 \in \Psi$ are arbitrary, this proves the lemma.
\end{proof}

\begin{lem}\label{lem:dense1}
Let $\psi \in \langle \Psi \rangle$, $f\in L^\infty(X)$ and define $f' := f - f\circ \psi$. Then  $\sA[f'|\cF_r]$ converges pointwise a.e. to $\EE[f'| \cI(\cR)]=0$ as $r\to\infty$.
\end{lem}

\begin{proof}
For a.e. $x \in X$, the previous lemma implies
\begin{eqnarray*}
\lim_{r\to\infty}\big|\sA[f'|\cF_r](x)\big| &=& \lim_{r\to\infty} \Big|\frac{1}{|\cF_r(x)|} \sum_{x' \in \cF_r(x)} f(x') - f(\psi(x'))\Big|\\
 &\le& 2||f||_\infty\lim_{r\to\infty} \frac{|\cF_r(x) \Delta \psi(\cF_r(x))|}{|\cF_r(x)|} =0.
 \end{eqnarray*}
By definition, $\EE[f|\cI(\cR)] = \EE[f\circ\psi | \cI(\cR)]$. Hence $\EE[f'|\cI(\cR)]=0$ a.e.. This proves the lemma.


\end{proof}



\begin{lem}
Let $f$ be a measurable function on $X$ such that for every $\psi \in \langle \Psi\rangle$, $f=f\circ \psi$ a.e. Then $f$ is $\cR$-invariant, namely $f(x)=f(x')$ for $\mu\times c$-a.e. $(x,x') \in \cR$.
\end{lem}

\begin{proof}
For each $\psi \in \langle \Psi\rangle $, let 
$$X_\psi:=\{x \in X:~ f(x) \ne f\circ\psi(x)\}.$$
Since $\Psi$ is countable, $\langle \Psi \rangle $ is also countable and 
$$\mu\Big(\bigcup_{\psi \in \langle \Psi \rangle} X_\psi \Big)=0.$$
By definition if $x \notin \bigcup_{\psi \in \langle \Psi \rangle} X_\psi$, then $f(x) = f(\psi(x))$ for all $\psi \in \langle \Psi \rangle$. But this implies $f(x)=f(x')$ for $\mu \times c$-a.e. $(x,x') \in \cR$, since the union of the graphs of $\psi\in \Psi$ coincides with $\cR$ up to a set of $\mu\times c$-measure zero. 
\end{proof}

\begin{proof}[Proof of Theorem \ref{thm:dense}]
Let $\cI \subset L^2(X)$ be the space of $\cR$-invariant $L^2$ functions. That is, $f \in \cI$ if and only if $f(x)=f(x')$ for a.e. $\big( x,x'\big) \in \cR$. Let $\cG \subset L^2(X)$ be the space of all functions of the form $f-f\circ \psi$ for $f\in L^\infty(X)$ and $\psi \in \langle \Psi \rangle$. We claim that the span of $\cI$ and $\cG$ is dense in $L^2(X)$. To see this, let $f_*$ be a function in the orthocomplement of $\cG$. Denoting the $L^2$ inner product by $\langle \cdot,\cdot \rangle$, we have
$$0=\langle f_*, f-f\circ\psi\rangle = \langle f_*,f\rangle - \langle f_*, f\circ \psi\rangle = \langle f_*,f\rangle - \langle f_* \circ \psi^{-1}, f\rangle = \langle f_*-f_*\circ\psi^{-1},f\rangle$$
for any $f \in L^\infty(X)$ and  $\psi \in \langle \Psi \rangle$. Since $L^\infty(X)$ is dense in $L^2(X)$, we have $f_*=f_*\circ \psi^{-1}$ for all $\psi \in \langle \Psi \rangle$. So the previous lemma implies $f_*$ is $\cR$-invariant; i.e., $f_* \in \cI$. This implies $\cI+\cG$ is dense in $L^2( X)$ as claimed. 

By Lemma \ref{lem:dense1} for every $f\in \cI + \cG$, $\sA[f|\cF_r]$ converges pointwise a.e. to $\EE[f|\cI(\cR)]$. Since $\cI+\cG$ is dense in $L^2( X)$, which is dense in $L^1( X)$, the first statement follows. The second is similar.
\end{proof}

\subsection{Maximal inequality: the regular case}
To prove the maximal inequality for a regular F\o lner family, we begin with the following basic covering argument, motivated by the classical case.  
\begin{lem}\label{lem:covering}
Suppose $\cF$ satisfies the regularity condition with constant $C_{reg}>0$. Let $\rho:Y \to \II$ be a bounded measurable function where $Y \subset X$ is Borel. Then there exists a measurable set $Z \subset Y$ such that the family of sets $\cF_{\rho(z)}(z)$, $z\in Z$ satisfy 
\begin{enumerate}
\item for all $z_1\ne z_2 \in Z$, $\cF_{\rho(z_1)}(z_1) \cap \cF_{\rho(z_2)}(z_2) = \emptyset$, namely the family is disjoint ;
\item The union of the sets in the family covers at least a fixed fraction of the measure of $Y$: 
$$C_{reg}\mu \Big( \bigcup_{z\in Z} \cF_{\rho(z)}(z) \Big) \ge  \mu (Y)\,. $$

\end{enumerate}
\end{lem}

\begin{proof}
Let $T:X \to \RR$ be an injective Borel function. We will use $T$ to break `ties' in what follows. 

If $Y' \subset Y$ is a Borel set then we let $M(Y')\subset Y'$  be the set of all `maximal' elements of $Y'$. Precisely, $y_1 \in M(Y')$ if $y_1 \in Y'$ and for all $y_2\in Y'$ {\bf different from $y_1$ } either
\begin{enumerate}
\item  $\cF_{\rho(y_1)}(y_1) \cap \cF_{\rho(y_2)}(y_2) =\emptyset$,
\item $\rho(y_1) > \rho(y_2)$ or
\item  $\cF_{\rho(y_1)}(y_1) \cap \cF_{\rho(y_2)}(y_2) \ne \emptyset$, $\rho(y_1) = \rho(y_2)$ and $T(y_1) > T(y_2)$.
\end{enumerate}
Because $\rho$ is bounded, the equivalence relation has countable classes, and $\cF$ is regular it follows that 
for any $y_1$, the set of $y_2$ with $\cF_{\rho(y_1)}(y_1) \cap \cF_{\rho(y_2)}(y_2) \ne \emptyset$
is finite. Thus in case 3) there exists a point $y_1$ with $T(y_1)$ maximal, so that if $Y'$ is non-empty then $M(Y')$ is nonempty. 

Let $Y_0:=Y$ and $M_0:=M(Y_0)$. Assuming that $Y_n, M_n \subset Y$ have been defined, let 
$$Y_{n+1}:=\{y \in Y:~ \cF_{\rho(y)}(y) \cap \cF_{\rho(z)}(z)=\emptyset ~\forall z \in M_n\}$$
and $M_{n+1}:=M(Y_{n+1})$.  Let 
$$Z:= \bigcup_n M_n, \quad ~\tZ := \bigcup_{z\in Z} \cF_{\rho(z)}(z).$$
By construction, for all $z_1\ne z_2 \in Z$, $\cF_{\rho(z_1)}(z_1) \cap \cF_{\rho(z_2)}(z_2) = \emptyset$. Also
$$Y \subset W:=\bigcup_{z \in Z} \bigcup_{ r \le \rho(z)} \cF_r^{-1}\cF_{\rho(z)}(z).$$
So it suffices to show $C_{reg} \mu(\tZ) \ge \mu(W)$. 

Define $K:\cR \to \RR$ by 
$$K(x,y)=\abs{\bigcup_{ r \le \rho(z)} \cF_r^{-1}\cF_{\rho(z)}(z)}^{-1}$$ 
if there is a (necessarily unique) $z \in Z$ such that $y \in \cF_{\rho(z)}(z)$ and $x \in \bigcup_{ r \le \rho(z)} \cF_r^{-1}\cF_{\rho(z)}(z)$. Let $K(x,y)=0$ otherwise. Because $\mu \times c|_\cR = c\times \mu|_\cR$,
\begin{eqnarray*}
\mu(\tZ) &=& \int \sum_{x\in [y]} K(x,y) ~d\mu(y) = \int \sum_{y\in [x]} K(x,y) ~d\mu(x).
\end{eqnarray*}
Observe that $\sum_{y\in [x]} K(x,y)=0$ unless $x \in W$ in which case 
$$\sum_{y\in [x]} K(x,y) \ge \frac{|\cF_{\rho(z)}(z)|}{|\bigcup_{ r \le \rho(z)} \cF_r^{-1}\cF_{\rho(z)}(z)|} \ge C_{reg}^{-1}$$
where $z \in Z$ is any element such that $x \in \bigcup_{ r \le \rho(z)} \cF_r^{-1}\cF_{\rho(z)}(z)$. Thus
$$\mu(\tZ)= \int \sum_{y\in [x]} K(x,y) ~d\mu(x) \ge C_{reg}^{-1}\mu(W)$$
which implies the lemma.
\end{proof}

We can now prove the weak-type $(1,1)$-maximal inequality (and measurability of the maximal function) for a regular family.
\begin{lem}\label{lem:doubling}
Suppose that $\cF$ is regular with regularity constant $C_{reg}>0$. Then for any $f \in L^1(X)$ and any $t>0$,
$$\mu\left(\left\{ x\in X:~\sM[f|\cF]>t \right\}\right) \le \frac{C_{reg}||f||_1}{t}.$$
\end{lem}

\begin{proof}
For $n>0$, let 
$$\sM_n[f|\cF](x) := \max_{0 < r \le n} \sA[|f|| \cF_r](x).$$
Given $x$, for $s\le n$ the family $\cF_s(x)$ comprises of a finite number of finite subsets of $[x]$, by our standing assumption on the family $\cF$, and furthermore $(x,s)\mapsto \cF_s(x)$ is Borel. Hence  $\sM_n[f|\cF](x)$ is measurable, and since $\sM[f|\cF](x)=\lim_{n\to \infty}\sM_n[f|\cF](x)$, so is the $\sM[f|\cF](x)$. 

Let now $D_{n,t}:= \{x \in X:~\sM_n[f|\cF](x)>t\}$. It suffices to show that $\mu(D_{n,t}) \le \frac{C_{reg}||f||_1}{t}$ for each $n>0$.

Let $\rho:D_{n,t} \to \II$ be a Borel function such that $\sA[|f| | \cF_{\rho(x)}](x)>t$ and $\rho(x) \le n ~\forall x \in D_{n,t}$. Let $Z \subset D_{n,t}$ be the subset given by the previous lemma where $Y=D_{n,t}$. As before let $\tZ=\cup\{ \cF_{\rho(z)}(z):~z\in Z\}$.  The previous lemma implies
$ \mu(D_{n,t})  \le C_{reg}\mu(\tZ).$

The disjointness property of $Z$ implies that for every $z \in \tZ$ there exists a unique element $\pi(z) \in Z$ with $z \in \cF_{\rho(\pi(z))}(\pi(z))$. By definition of $\rho$,
\begin{eqnarray*}
 \mu(D_{n,t}) \le C_{reg}  \mu(\tZ) \le \frac{C_{reg}}{ t} \int_{\tZ} \sA[|f| | \cF_{\rho(\pi(z))}](\pi(z))  ~d\mu(z).
 \end{eqnarray*}
Let $K:\cR \to \RR_+$ be the function 
$$K(y,z)=\frac{|f(y)|}{|\cF_{\rho(\pi(z))}(\pi(z))|}$$
if $z \in \tZ$ and $y \in  \cF_{\rho(\pi(z))}(\pi(z))$, and let $K(y,z)=0$ otherwise. Note that for a given $y\in \tZ$, the number of elements $z\in [y]$ such that $y\in \cF_{\rho(\pi(z))}(\pi(z))$ is precisely $\abs{ \cF_{\rho(\pi(z))}(\pi(z))}$.  Since $\mu \times c|_{\cR} = c \times \mu|_{\cR}$, we conclude 
$$
\int_{y\in X} \sum_{z\in [y]} K(y,z) ~d\mu(y)   = \int_{y\in \tZ} |f(y)| ~d\mu(y)= $$
 $$=   \int_{z\in X} \sum_{y\in [z]} K(y,z) ~d\mu(z) =\int_{\tZ}\sA[|f| | \cF_{\rho(\pi(z))}](\pi(z)) ~d\mu(z).
$$
So
\begin{eqnarray*}
 \mu(D_{n,t}) \le \frac{C_{reg}}{ t} \int_{\tZ} \sA[|f| | \cF_{\rho(\pi(z))}](\pi(z))~d\mu(z) = \frac{C_{reg}}{t}\int_{\tZ} |f(y)| ~d\mu(y) \le \frac{C_{reg}||f||_1}{ t}, 
    \end{eqnarray*}

and the proof of the maximal inequality is complete. 
\end{proof}

\subsection{Maximal inequality: the tempered case}

This subsection completes the proofs of Theorems \ref{thm:maximal} and \ref{thm:pointwise} using \cite{We03} as a model. Having considered the regular case in the previous lemma, it suffices to assume $\cF$ is asymptotically invariant, uniform and tempered.



\begin{lem}\label{lem:average}
Suppose $\cF$ is uniform with uniformity constant $C_u>0$. If $f\in L^1(X)$ with $f\ge 0$ and $r>0$ then
$$C_u^{-1}\int f(x)~d\mu(x)\le \int \sA[f|\cF_r](x)~d\mu(x) \le C_u \int f(x)~d\mu(x).$$
\end{lem}
\begin{proof}
Define a function $F$ on $\cR$ by $F(x,y) := \frac{f(y)}{|\cF_r(x)|}$ if $y \in \cF_r(x)$ and $F(x,y):=0$ otherwise. Because $\mu \times c|_\cR =c\times \mu|_\cR$,
\begin{eqnarray*}
\int \sA[f|\cF_r](x)~d\mu(x)&=& \int F(x,y) ~d\mu\times c(x,y) = \int F(x,y) ~dc\times \mu(x,y)\\
&=& \int f(y) \sum_{x \in \cF_r^{-1}(y)} |\cF_r(x)|^{-1}~d\mu(y).
\end{eqnarray*}
Let $a_r,b_r$ be the constants in the definition of uniformity. Then $a_r \le |\cF_r^{-1}(y)| \le b_r$ and $a_r\le |\cF_r(x)|\le b_r$ for a.e. $x,y \in X$. Therefore,
$$C_u^{-1}\le a_r/b_r \le \sum_{x \in \cF_r^{-1}(y)} |\cF_r(x)|^{-1} \le b_r/a_r \le C_u.$$
These inequalities and the equality above imply the lemma.
\end{proof}
We now turn to establish the important fact that in the uniform case, asymptotic invariance under a generating set implies asymptotic invariance (in mean) under the entire group of inner automorphisms. 
\begin{lem}\label{lem:invariance}
If $\cF$ is uniform and asymptotically invariant then for every $\phi \in \Inn(\cR)$,
$$\lim_{r\to\infty} \int \frac{|\cF_r(x) \Delta \phi(\cF_r(x))|}{|\cF_r(x)|}~d\mu(x) =0.$$
\end{lem}
\begin{proof}
Because 
\begin{eqnarray*}
|\cF_r(x) \Delta \phi(\cF_r(x))| &=& |\cF_r(x) \setminus \phi(\cF_r(x))| + |\phi(\cF_r(x)) \setminus \cF_r(x)|\\
&=&  |\cF_r(x) \setminus \phi(\cF_r(x))| + |\cF_r(x) \setminus \phi^{-1}(\cF_r(x))|
\end{eqnarray*}
and $\phi \in \Inn(\cR)$ is arbitrary, it suffices to show 
$$\lim_{r\to\infty} \int \frac{|\cF_r(x) \setminus \phi(\cF_r(x))|}{|\cF_r(x)|}~d\mu(x) =0.$$

Let $\Psi \subset \Inn(\cR)$ be a countable generating set witnessing the asymptotic invariance. So, $\Psi$ generates $\cR$ and for a.e. $x\in X$,
$$\lim_{r\to\infty}  \frac{|\cF_r(x) \Delta \psi(\cF_r(x))|}{|\cF_r(x)|} = 0 \quad \forall \psi \in \Psi.$$
By Lemma \ref{lem:group}, we may assume, without loss of generality, that $\Psi$ is a subgroup of $\Inn(\cR)$. Because $\Psi$ generates $\cR$ this means that for $\mu \times c$-a.e. $(x,y) \in \cR$, there is a $\psi \in \Psi$ such that $\psi(x)=y$. Because $\Psi$ is countable, this implies that there is a Borel partition $\{X_i\}_{i=1}^\infty$ of $X$ and elements $\psi_i \in \Psi$ such that $\phi(x)=\psi_i(x)$ for a.e. $x\in X_i$. 

Let $\epsilon>0$. Choose $N>0$ so that $\mu(\cup_{i=1}^N X_i) \ge 1-\epsilon$. Let $Y = \cup_{i>N} X_i$, so $\mu(Y)\le \epsilon$. Let $U$ be the subset function  $U(x)=\set{\psi_i(x)\,;\, 1\le i \le N}$. Recall that we have defined in \S 2.1 the product of two arbitrary subset functions, and therefore the expression 
$U\cF_r(x)=\{\psi_i(y)\,;\, y\in \cF_r(x), 1\le i\le N\}$ makes sense and is also a subset function. Using the foregoing pointwise convergence results established for $\psi\in \langle\Psi\rangle$, Lebesgue's bounded convergence theorem implies
$$\lim_{r\to\infty} \int \frac{|\cF_r(x) \setminus U\cF_r(x)|}{|\cF_r(x)|}~d\mu(x) =0.$$
However,
$$|\cF_r(x) \setminus \phi(\cF_r(x))| \le |\cF_r(x) \setminus U\cF_r(x)|+ | \cF_r(x) \cap Y|.$$
Thus
\begin{eqnarray*}
\lim_{r\to\infty} \int \frac{|\cF_r(x) \setminus \phi(\cF_r(x))|}{|\cF_r(x)|}~d\mu(x)  &\le& \lim_{r\to\infty} \int \frac{|\cF_r(x) \setminus U\cF_r(x)|}{|\cF_r(x)|}~d\mu(x)  +  \lim_{r\to\infty} \int \frac{  | \cF_r(x) \cap Y| }{  | \cF_r(x) |} ~d\mu(x)\\
&=& 0 +  \lim_{r\to\infty} \int \sA[1_Y|\cF_r]~d\mu(x) \le C_u \mu(Y) \le C_u \epsilon
\end{eqnarray*}
where $C_u$ is the uniformity constant of $\cF$. The second to last inequality above follows from Lemma \ref{lem:average}. Since $\epsilon>0$ is arbitrary, this implies the lemma.

\end{proof}
We can formulate the following useful fact which will be used in the proof of the maximal inequality in the tempered case. 

\begin{lem}\label{lem:regular}
If $\cF$ is uniform and asymptotically invariant and $U$ is a Borel subset function with $1 \le |U(x)|$ for a.e. $x$ and $|U| \in L^\infty(X)$ then,
$$ \lim_{r\to\infty} \int \frac{ |  U\cF_r(x)| }{|\cF_r(x)|} ~d\mu(x) = 1.$$
\end{lem}

\begin{proof}

Let $E=\{(x,y) \in \cR:~ x \in U(y)$ or $y\in U(x)\}$. Because $U$ is bounded, this a bounded degree graph. By \cite{KST99}, this implies that the Borel edge-chromatic number of $(X,E)$ is finite. That is, there exists a Borel map $\alpha:E \to A$ (where $A$ is a finite set) such that if $(x,y),(y,z)  \in E$ and $x\ne z$ then $\alpha((x,y)) \ne \alpha((y,z))$. We can also assume without loss of generality that $\alpha(x,y)=\alpha(y,x)$. 

For each element $a \in A$, define $\phi_a:X\to X$ as follows. If $x \in X$ and there is a $y \ne x $ such that $(x,y) \in E$ and $\alpha(x,y)=a$ then define $\phi_a(x)=y$ and $\phi_a(y)=x$. Otherwise, let $\phi_a(x)=x$. Then $\phi$ is a Borel bijection and $\phi_a \in \Inn(\cR)$. 

So we have proven that there is a finite collection of automorphisms $\phi_1,\ldots, \phi_m \in \Inn(\cR)$ such that for a.e. $x\in X$,
$$U(x) \subset \bigcup_{i=1}^m \phi_i(x).$$
Lemma \ref{lem:invariance} implies that for every $i$,
$$\lim_{r\to\infty} \int \frac{|\cF_r(x) \Delta \phi_i(\cF_r(x))|}{|\cF_r(x)|} ~d\mu(x)=0.$$
Since this is true for every $i$, it follows that
$$\lim_{r\to\infty} \int \frac{|\cF_r(x) \Delta U\cF_r(x))|}{|\cF_r(x)|} ~d\mu(x)=0$$
which implies the lemma.

\end{proof}

 We now state the following combinatorial result from \cite{We03} together with its proof,  which will serve as a model in the more complicated set-up of measured equivalence relations.

\begin{lem}[Basic Lemma \cite{We03}]
Let $\Omega$ be a countable set, $V_1,\ldots, V_m \subset \Omega$ be non-empty finite subsets, $\kappa$ be a positive measure on $\Omega$ and $C_u\ge 1, \lambda>0$ be constants. Suppose
\begin{enumerate}
\item $\frac{|V_i|}{|V_j|} \le C_u$  for every $i,j$.
\item $\kappa(V_i) \ge \lambda |V_i|$ for every $i$.
\item $\sum_{i=1}^m 1_{V_i}(\omega) \le C_u |V_1|$ for every $\omega \in \Omega$.
\end{enumerate}
Then there is a subset $I \subset \{1,\ldots, m\}$ such that
\begin{enumerate}
\item $\kappa( \cup_{i\in I} V_i ) \ge \frac{\lambda m }{4 C_u^2}$.
\item $\kappa( \cup_{i\in I} V_i ) \ge \frac{\lambda |I|  |V_1| }{4 C_u^2}$.
\end{enumerate}
\end{lem}

\begin{proof}

Beginning with $i(1)=1$ inductively define $i(k+1)$ to be the least integer $\le m$, greater than $i(k)$, such that
$$\kappa\left( V_{i(k+1)} \setminus \bigcup_{1\le j \le k} V_{i(j)} \right) \ge \frac{1}{2} \kappa(V_{i(k+1)})$$
is such an integer exists, otherwise stop and call $\{i(1),\ldots, i(k)\}=:I$. We distinguish two cases. 

\noindent {\bf Case 1}. $|I| \ge \frac{m}{2|V_1|}$. In this case clearly,
$$\kappa(\cup_{i\in I} V_{i}) \ge \frac{1}{2}\sum_{i \in I} \kappa(V_i)  \ge \frac{|I| \lambda  |V_1|}{2 C_u} \ge \frac{ \lambda m  }{4 C_u}.$$

\noindent {\bf Case 2}. $|I| < \frac{m}{2|V_1|}$. Let $I^c =\{1,\ldots, m\} \setminus I$. By definition of $I$, if $j \in I^c$ then
$$\kappa\left( V_j \cap \bigcup_{i\in I} V_i \right) \ge \frac{1}{2} \kappa(V_j).$$
Sum over all $j \in I^c$ and use hypothesis 3 to obtain
$$\frac{1}{2}\sum_{j \in I^c} \kappa(V_j) \le \sum_{j \in I^c} \kappa\left( V_j \cap \bigcup_{i\in I} V_i \right) \le C_u |V_1| \kappa\left( \bigcup_{i\in I} V_i \right).$$
Now use hypothesis 2 and divide by $C_u|V_1|$ to obtain
$$\kappa\left( \bigcup_{i\in I} V_i \right) \ge \frac{1}{2C_u|V_1|} \sum_{j\in I^c} \kappa( V_j) \ge \frac{|I^c| \lambda}{2 C_u^2} \ge \frac{(m-\frac{m}{2|V_1|}) \lambda}{2C_u^2} = \frac{1}{2 C_u^2}(1-2^{-1}|V_1|^{-1}) m\lambda.$$ 
Because $|V_1|\ge 1$, $\frac{1}{2}(1-2^{-1}|V_1|^{-1}) \ge 1/4$. So this implies
$$\kappa\left( \bigcup_{i\in I} V_i \right) \ge \frac{1}{4 C_u^2}m\lambda.$$ 
This proves the first conclusion. The second one follows from the inequality above and the hypothesis $|I| < \frac{m}{2|V_1|}$.
\end{proof}

For the next proposition, we let $\Omega$ be a countable set and $\{V_i\}_{i=1}^N$ a sequence of subset functions on $\Omega$. Thus each $V_i$ is a map $V_i:\Omega \to 2^\Omega$. We define the inverse $V_i^{-1}:\Omega \to 2^\Omega$ by $V_i^{-1}(y)=\{x\in \Omega:V_i(x) \ni y\}$. We also define  the product, union, intersection and difference of subset functions as in \S 2.1, which considers the special case of subset functions for equivalence relations.

\begin{prop}\label{prop:benjy}
Let $\Omega$ be a countable set, $I_1,\ldots, I_N \subset \Omega$ be pairwise disjoint finite subsets, $\{V_{i}:~ 1\le i \le N\}$ a collection of subset functions of $\Omega$, $\kappa$ be a positive measure on $\Omega$ and $C_t, C_u, \lambda>0$ be constants. Suppose
\begin{enumerate}
\item $\frac{|V_{i}(\omega)|}{|V_{i}(\omega')|} \le C_u$ for every $i$ and every $\omega, \omega' \in \Omega$.
\item $\kappa(V_{i}(\omega)) \ge \lambda |V_{i}(\omega)|$ for every $i$ and every $\omega \in I_i$.
\item $|V_i^{-1}|(\omega) \le C_u V_i(\omega)$ for every $i$ and $\omega \in \Omega$.
\item for every $j$, $|\cup_{i<j} V_i^{-1}V_j(\omega)| \le C_t |V_j(\omega)|$.

\end{enumerate}
Then 
$$\sum_{i=1}^N |I_i| \le \left(\frac{8C_u^2 + 8 C_tC_u^3}{\lambda}\right) \kappa\left(\cup_{i=1}^N \cup_{\omega \in I_i} V_i(\omega) \right).$$
\end{prop}

\begin{proof}
Without loss of generality, we may assume each $I_i$ is nonempty. For each $i$ with $1\le i \le N$, choose $\omega_i \in I_i$. We construct a partition $\{L,K\}$ of $\{1,\ldots, N\}$ and sets $D_i \subset I_i$ for $i\in L$ using the following algorithm.
\begin{description}
\item[Step 1] Apply the Basic Lemma to the collection $\{V_{N}(\omega):~\omega \in I_N\}$ to obtain a set $D_N \subset I_N$ such that
\begin{enumerate}
\item $\kappa( \cup_{\omega \in D_N} V_{N}(\omega) ) \ge \frac{\lambda |I_N|}{4 C_u^2}$;
\item $\kappa( \cup_{\omega \in D_N} V_{N}(\omega) ) \ge \frac{\lambda |D_N| |V_{N}(\omega_N)|}{4 C_u^2}$.
\end{enumerate}
It is convenient to rewrite these inequalities in the form:
\begin{enumerate}
\item $|I_N| \le \frac{4C_u^2}{\lambda} \kappa( \cup_{\omega \in D_N} V_{N}(\omega) ) $;
\item $|D_N| |V_{N}(\omega_N)| \le \frac{4C_u^2}{\lambda} \kappa( \cup_{\omega \in D_N} V_{N}(\omega) ) $.
\end{enumerate}
\item[Step 2] Let $L:=\{N\}$, $K:=\emptyset$, $i:=1$.
\item[Step 3] If $i=N$ then stop.
\item[Step 4] Let $I'_{N-i}$ be the set of $\omega \in I_{N-i}$ such that $V_{N-i}(\omega)$ is disjoint from $\cup\{ V_{k}(\omega'):~ k \in L, \omega' \in D_k\}$.
\item[Step 5] If $|I'_{N-i}| \ge \frac{1}{2}|I_{N-i}|$ then 
\begin{enumerate}
\item Set $L:=L \cup \{N-i\}$,
\item Apply the Basic Lemma to obtain a set $D_{N-i} \subset I'_{N-i}$ such that
\begin{enumerate}
\item $|I_{N-i}| \le \frac{8C_u^2}{\lambda} \kappa( \cup_{\omega \in D_{N-i}} V_{N-i}(\omega) ) $;
\item $|D_{N-i} | |V_{N-i}(\omega_{N-i}) | \le \frac{4C_u^2}{\lambda} \kappa( \cup_{\omega \in D_{N-i}} V_{N-i}(\omega) )$.
\end{enumerate}
\end{enumerate}
\item[Step 6] If $|I'_{N-i}| < \frac{1}{2}|I_{N-i}|$ then set $K:=K \cup \{N-i\}$.
\item[Step 7] Set $i:=i+1$ and go to Step 3.
\end{description}
This algorithm produces a partition $\{L,K\}$ of $\{1,\ldots, N\}$ and subsets $D_i \subset I_i$ for $i\in L$ such that
\begin{enumerate}
\item if, for $i \in L$, $H_i:= \cup\{ V_{i}(\omega):~ \omega\in D_i\}$ then $H_i \cap H_k = \emptyset$ for all $i\ne k$;
\item $|I_{i}| \le \frac{8C_u^2}{\lambda} \kappa( \cup_{\omega \in D_{i}} V_{i}(\omega) )$ for all $i\in L$;
\item $|D_{i} | |V_{i}(\omega_i) | \le \frac{4C_u^2}{\lambda} \kappa( \cup_{\omega \in D_{i}} V_{i}(\omega) )$ for all $i\in L$.
\end{enumerate}
The first two conditions above imply
\begin{eqnarray*}
\sum_{i \in L} |I_i| &\le& \sum_{i \in L} \frac{8C_u^2}{\lambda} \kappa( \cup_{\omega \in D_{i}} V_{i}(\omega) )= \frac{8C_u^2}{\lambda} \kappa( \cup_{i \in L} H_i ).
\end{eqnarray*}

Also if $k\in K$ then there exists a set $I''_k \subset I_k$ such that $|I''_k| \ge \frac{1}{2}|I_k|$ and for every $\omega \in I''_k$, $V_{k}(\omega)$ is not disjoint from $\cup\{ H_i:~ i>k, i \in L\}$. Therefore, $\omega \in V_k^{-1}V_j(\omega')$ for some $j>k$ with $j\in L$ and some $\omega' \in D_j$. 
Because $\{I_i\}_{i=1}^N$ are pairwise disjoint, hypothesis 4 implies 
\begin{eqnarray*}
\sum_{k \in K} |I_k| &\le& 2 \sum_{k \in K} |I''_k|\\ 
&\le& 2 |\cup \{ V_k^{-1}V_j(\omega)  :~ j \in L, j>k,  \omega \in D_j\}|\\
&\le& 2 \sum_{j\in L, \omega \in D_j}\left| \bigcup_{i<j} V_i^{-1} V_j(\omega)\right|\\
&\le& 2 C_t \sum_{j \in L, \omega \in D_j} |V_{j}(\omega)| \\
&\le& 2 C_t C_u \sum_{j \in L} |D_j| |V_{j}(\omega_j)|\le 2 C_t C_u \frac{4C^2_u}{\lambda} \kappa( \cup_{i \in L} H_i).
\end{eqnarray*}
Thus
\begin{eqnarray*}
\sum_{i=1}^N |I_i| &=& \sum_{i \in L} |I_i|  + \sum_{k \in K} |I_k|\le \frac{8C_u^2 + 8 C_tC_u^3}{\lambda} \kappa( \cup_{i \in L} H_i)
\end{eqnarray*}
which implies the result.
\end{proof}


\begin{proof}[Completion of the proof of Theorem \ref{thm:maximal}]
By Lemma \ref{lem:doubling}, it suffices to assume $\cF$ is asymptotically invariant, uniform and tempered. Note that temperedness is defined for sequences only, so the measurability of the maximal function is obvious in this case. For $f\in L^1(X)$ define $\sM_N[f ] := \sup_{r\le N} \sA[| f| |\cF_r]$. It suffices to prove the existence of a constant $C>0$ such that for every $\lambda>0$, every $N>0$ and every $f\in L^1(X)$ with $f\ge 0$,
$$\mu( \{x\in X:~ \sM_N[f](x) \ge \lambda \}) \le \frac{C \|f\|_1}{\lambda}.$$
So fix $N>0, \lambda>0$ and $f\in L^1(X)$ with $f\ge 0$. Let 
$$E_N:=\{x\in X:~ \sM_N[f](x) \ge \lambda\}.$$
For $R>0$, let $H(N,R)$ be the subset function
$$H(N,R)(x):=E_N \cap \cF_R(x).$$
Let $1_{E_N}$ be the indicator function of $E_N$. Observe that $\sA[1_{E_N}| \cF_R](x) = \frac{ |H(N,R)(x)|}{|\cF_R(x)|}$. By Lemma \ref{lem:average},
\begin{eqnarray}\label{eqn:H}
\mu(E_N) \le C_u \int \frac{ |H(N,R)(x)|}{ |\cF_R(x)|}~d\mu(x).
\end{eqnarray}
Let $H'_{N,R}$ be the subset function
$$H'(N,R)(x):=\{ y\in X:~ \exists n \le N, \sA[f|\cF_n](y) \ge \lambda, \cF_n(y) \subset \cF_R(x)\}.$$

To apply Proposition \ref{prop:benjy}, let $\Omega:=[x]$, the equivalence class of $x$.  Let $\kappa$ be the measure on $\Omega$ determined by $\kappa(\{y\}) := f(y)$ (for $y\in \Omega$). For each $y \in H'(N,R)(x)$, let $k(y)$ be the smallest number such that $\cF_{k(y)}(y)$ satisfies $\sA[f|\cF_{k(y)}](y) \ge \lambda$, $\cF_{k(y)}(y) \subset \cF_R(x)$. For each $1\le i\le N$, let $I_i$ be the set of all $y \in H'(N,R)$ such that $i=k(y)$. Let $V_i(y):=\cF_i(y)$ for $1\le i \le N$ and $y\in \Omega$.  It is easy to check that because $\cF$ is uniform and tempered the hypotheses of Proposition \ref{prop:benjy} are satisfied. The conclusion implies:
$$|H'(N,R)(x)| \le \frac{C}{\lambda} \sum_{y \in \cF_R(x)} f(y)$$
where $C=8C_u^2 + 8C_tC_u^3$. Divide both sides by $\cF_R(x)$ and integrate over $x$ to obtain:
\begin{eqnarray}\label{eqn:H'}
\int \frac{|H'(N,R)(x)|}{|\cF_R(x)|} ~d\mu(x) \le \frac{C}{\lambda} \int \sA[f |\cF_R](x)~d\mu(x) \le \frac{CC_u}{\lambda} \|f\|_1.
\end{eqnarray}
The last inequality follows from Lemma \ref{lem:average}.

Let $U(N)$ and $S(N,R)$ be the subset functions
$$U_N(x) = \cup_{t \le N} \cF_t(x), \quad S(N,R)(x):=\{y \in \cF_R(x):~ U_N(y) \nsubseteq \cF_R(x)\}.$$
Observe that 
$$H(N,R) \setminus H'(N,R) \subset S(N,R) \subset U_N^{-1}(U_N\cF_R \setminus \cF_R).$$ 
By Lemma \ref{lem:regular},
$$\lim_{R\to \infty} \int \frac{|U_N\cF_R(x)\setminus \cF_R(x)|}{|\cF_R(x)|}~d\mu(x) = 0.$$
Because $\cF$ is uniform, the function $x \mapsto |U_N^{-1}(x)|$ is essentially bounded. Therefore,
$$\lim_{R\to \infty} \int \frac{|U_N^{-1}(U_N\cF_R \setminus \cF_R)(x)|}{|\cF_R(x)|}~d\mu(x) = 0.$$
Since  $H(N,R) \setminus H'(N,R) \subset U_N^{-1}(U_N\cF_R \setminus \cF_R)$, it follows that
$$\lim_{R\to \infty} \int \frac{|H(N,R)(x) \setminus H'(N,R)(x)|}{|\cF_R(x)|}~d\mu(x) = 0.$$
Since $H'(N,R) \subset H(N,R)$, equations (\ref{eqn:H}), (\ref{eqn:H'}) now imply
\begin{eqnarray*}
\mu(E_N) &\le&\lim_{R\to\infty} C_u \int \frac{ |H(N,R)(x)|}{ |\cF_R(x)|}~d\mu(x)=\lim_{R\to\infty} C_u \int \frac{ |H'(N,R)(x)|}{ |\cF_R(x)|}~d\mu(x)\le \frac{CC_u^2}{\lambda} \|f\|_1.
\end{eqnarray*}
Because $f,N,\lambda$ are arbitrary, this implies the Theorem.

\end{proof}

\noindent{\it Completion of the proof of Theorem \ref{thm:pointwise}.}
\begin{lem}
If $\cF$ is any family of subset functions satisfying the conclusions of Theorems \ref{thm:dense} and \ref{thm:maximal} (i.e., there exists a dense set of good functions and the weak $(1,1)$-type maximal inequality is satisfied) then $\cF$ is a (restricted) pointwise ergodic family in $L^1$.  Namely, for every $f\in L^1(X,\mu)$, $\sA[f|\cF_r]$ converges pointwise a.e. to $\EE[f|\cI(\cR)]$ as $r\to\infty$.
\end{lem}

\begin{proof}
Let $f\in L^1(X)$. We will show that $\{\sA[f|\cF_r]\}_{r>0}$ converges pointwise a.e. to $\EE[f|\cI(\cR)]$. After replacing $f$ with $f-\EE[f|\cI(\cR)]$ if necessary we may assume that $\EE[f|\cI(\cR)]=0$ a.e..

For $t>0$, let $E_t :=\{x \in X:~ \limsup_{r\to\infty} |\sA[f| \cF_{r}](x)| \le t\}$. We will show that each $E_t$ has measure one. Let $\epsilon=\frac{t^2}{4}$. According to Theorem \ref{thm:dense}, there exists a function $f_1\in L^1(X)$ with $\|f-f_1\|_1 < \epsilon$ such that $\{\sA[f_1| \cF_{r}]\}_{r>0}$ converges pointwise a.e. to $0$ as $r\to\infty$. For any $r>0$,
\begin{eqnarray*}
|\sA[f| \cF_{r}]| &\le& |\sA[f-f_1| \cF_{r}]| + |\sA[f_1| \cF_{r}]|\le \sM[f-f_1|\cF] + |\sA[f_1| \cF_{r}]|.
\end{eqnarray*}
Let 
$$D:=\big\{x \in X:~\sM[f-f_1|\cF](x) \le \sqrt{\epsilon}\big\}.$$
 Since $\sA[f_1| \cF_{r}]$ converges pointwise a.e. to zero, for a.e. $x\in D$ there is an $N>0$ such that $r>N$ implies
\begin{eqnarray*}
|\sA[f| \cF_{r}](x)| &\le& \sM[f-f_1|\cF](x) + |\sA[f_1| \cF_{r}](x)| \le 2\sqrt{\epsilon} =t.
\end{eqnarray*}
Hence $D \subset E_t$ (up to a set of measure zero). By Theorem \ref{thm:maximal}, 
$$\mu(E_t) \ge \mu(D) \ge 1-C\epsilon^{-1/2}\|f-f_1\|_1>1-\sqrt{\epsilon} C=1-\frac{Ct}{2}.$$
For any $s<t$, $E_s \subset E_t$. So $\mu(E_t) \ge \mu(E_s) \ge 1-\frac{Cs}{2}$ for all $s<t$ which implies $ \mu(E_t)=1$. So the set $E:=\cap_{n=1}^\infty E_{1/n}$ has full measure. This implies the result.

 \end{proof}

Theorem \ref{thm:pointwise} follows immediately from the  lemma above and Theorems \ref{thm:dense} and   \ref{thm:maximal}.

\subsection{Extensions of Borel equivalence relations}
Together with the amenable equivalence relation on $X$ we must consider its class-bijective p.m.p. extensions, whose definition we now state.
For $(X,\cB,\mu,\cR)$ a discrete p.m.p. equivalence relation, a {\em class-bijective extension} of $(X,\mu,\cR)$ is a measured equivalence relation $(\tX,\tmu,\tilde{\cR})$ with a Borel map $\pi:\tX \to X$ satisfying the following.
\begin{enumerate}
\item $(\tX,\tmu,\tilde{\cR})$ is a discrete p.m.p. equivalence relation.
\item  $\pi_*\tmu=\mu$.
\item $(x,x') \in \tilde{\cR} \Rightarrow (\pi(x),\pi(x')) \in \cR$.
\item for a.e. $\tilde{\cR}$-equivalence class $[x] \subset \tilde{\cR}$, $\pi$ restricted to $[x]$ is a bijection onto the $\cR$-equivalence class $[\pi(x)]$.
\end{enumerate}

Suppose $\cF=\{\cF_r\}_{r \in \II}$ is a family of subset functions for $(X,\mu,\cR)$. Then we may lift this family as follows. Define $\tcF=\{\tcF_r\}_{r \in \II}$ by 
$$\tcF_r(x):=\pi^{-1}(\cF_r(\pi(x))) \cap [x] \quad \forall x \in \tX.$$
\begin{lem}
Let $P$ be a property in $\{$asymptotically invariant, uniform, regular, tempered$\}$. If $\cF$ has property $P$ then $\tcF$ also has property $P$.
\end{lem}


\begin{proof}
{\bf Case 1}. Suppose $P=$ asymptotically invariant.

Let $\Psi \subset \Inn(\cR)$ be a countable generating set witnessing the asymptotic invariance of $\cF$. This means that for a.e. $x\in X$ and $\psi\in\Psi$,
$$\lim_{r\to\infty} \frac{|\cF_r(x) \Delta \psi(\cF_r(x))| }{ |\cF_r(x)|} = 0.$$
For any $\psi \in \Psi$, define $\tpsi:\tX \to \tX$ by $\tpsi(x)=x'$ where $x' \in \tX$ is the unique element such that $(x,x') \in \tilde{\cR}$ and $\psi(\pi(x))=\pi(x')$. This is unique because the $\pi$ restricted to $[x]$ is a bijection onto its image. Let $\tPsi = \{\tpsi:~ \psi \in \Psi\}$. This is a countable set of inner automorphisms of $\tilde{\cR}$. Because $\pi$ restricted to each equivalence class is a bijection, for a.e. $x\in \tX$, $|\tcF_r(x)|=|\cF_r(x)|$ and $|\tcF_r(x) \Delta \tpsi(\tcF_r(x))|=|\cF_r(x) \Delta \psi(\cF_r(x))|$. So
$$\lim_{r\to\infty} \frac{|\tcF_r(x) \Delta \tpsi(\tcF_r(x))| }{ |\tcF_r(x)|} = \lim_{r\to\infty} \frac{|\cF_r(x) \Delta \psi(\cF_r(x))| }{ |\cF_r(x)|} =0.$$
The set $\tPsi$ is generating  because for a.e. $(x,x') \in \tilde{\cR}$ there is an element $\psi \in \Psi$ such that $\psi(\pi(x))=\pi(x')$, and this implies $\tpsi(x)=x'$. So we have verified all the conditions for the asymptotic invariance of $\tcF$.

{\bf Case 2}.  Suppose $P=$ regular.

Let $C_{reg}$ be a regularity constant for $\cF$.  Because $\pi$ restricted to any equivalence class is a bijection, for a.e. $x\in \tX$ and every $r>0$, 
\begin{eqnarray*}
\Big| \bigcup_{t \le r} \tcF_t^{-1}\tcF_r(x) \Big| &=&  \Big| \bigcup_{t \le r} \cF_t^{-1}\cF_r(\pi(x)) \Big|\le C_{reg}|\cF_r(\pi(x))| = C_{reg}|\tcF_r(x)|.
 \end{eqnarray*}
This proves $\tcF$ is regular.

The other cases: uniform and tempered can be handled similarly.
\end{proof}

Recall that we have defined  $\cF$ to be a {\em  pointwise ergodic family} in $L^p$ if for every class-bijective extension $(\tX,\tmu,\tilde{\cR})$ and every  $f\in L^p(\tX,\tmu)$, $\sA[f|\tcF_r]$ converges pointwise a.e. to $\EE[f|\cI(\tilde{R})]$. To establish this fact, first define for $f\in L^1(\tX)$, $\sM[f |\tcF] := \sup_r \sA[ |f| |\tcF_r]$ where $|f|$ denotes the absolute value of $f$. $\sM[\cdot | \tcF]$ is the {\em maximal operator} associated to the family of operators $\sA[\cdot |\tcF_r]$. As in the case of $\cF$, the maximal function $ \sM[f |\tcF] $ is Borel measurable, and using the Lemma above and Theorem  \ref{thm:maximal}, we conclude 


\begin{thm}[Weak $(1,1)$-type maximal inequality]\label{thm:maximal2}
Suppose that $\cF$ is either regular or (asymptotically invariant, uniform and tempered). Then there exists a constant $C>0$ such that for any class-bijective extension $(\tX,\tmu,\tilde{\cR})$ and any $f \in L^1(\tX,\tmu)$ and any $\lambda>0$,
$$\tmu\left(\left\{ x\in \tX:~\sM[f|\tcF]>\lambda \right\}\right) \le \frac{C||f||_1}{\lambda}.$$
In fact $C$ can be taken to be $8C_u^4(1 + C_tC_u)$ in the tempered case and $C_{reg}$ in the regular case.
\end{thm}

Now using the Lemma above, Theorem \ref{thm:maximal2} and Theorem \ref{thm:pointwise}, we conclude 
 
\begin{thm}\label{thm:pointwise2}
If $\cF$ is either (asymptotically invariant and regular) or (asymptotically invariant, uniform and tempered) then $\cF$ is a pointwise ergodic family in $L^1$. 
\end{thm}

Finally, we note the following fact, which is a standard consequence of the weak-type $(1,1)$-maximal inequality, and will be used below. 

\begin{thm}[Strong $L^p$ maximal inequality]\label{thm:maximalp}
Suppose that $\cF$ is either regular or (asymptotically invariant, uniform and tempered). Then for every $p>1$ there is a constant $C_p>0$ such that for any class-bijective extension $(\tX,\tmu,\tilde{\cR})$ and any $f \in L^p(\tX,\tmu)$, $ \|\sM[f|\tcF]\|_p  \le C_p \|f\|_p$. Also, there is a constant $C_1>0$ such that if $f\in (L\log L)(\tX,\tmu)$, then $\| \sM[f|\tcF] \|_1 \le C_1 \|f \|_{L\log L}$.
\end{thm}

\begin{proof}
This follows from the fact that $\sM[\cdot| \tcF]$ satisfies a weak (1,1)-type maximal inequality (by Theorem \ref{thm:maximal2} above) and standard interpolation arguments. Namely, 
since $\sM[f|\tcF]$ is of weak-type $(1,1)$ and is norm-bounded on $L^\infty$, it is norm-bounded in every $L^p$, $1 < p < \infty$ (see e.g. \cite[Ch. V, Thm 2.4]{SW}).
\end{proof}

\section{Maximal inequalities for general group actions}\label{sec:random}

Let $\Gamma$ be a countable group and $\Gamma \cc (B,\nu)$ an amenable action. Recall that the Maharam extension is the action $\Gamma \cc B\times \RR$ given by
\begin{equation}\label{maharam1}  g(b,t):=\left(gb, t + \ln \left( \frac{d\nu\circ g^{-1}}{d\nu}b) \right) \right).\end{equation}
Let $\theta$ be the measure on $\RR$ given by $d\theta(t)=e^t dt$.  The action above preserves the product measure $\nu \times \theta$. Let $T>0$, $I=[0,T]$ and let $\cR_I$ be the equivalence relation on $B\times I$ given by restricting the orbit equivalence relation on $B\times \RR$ (so $\cR_I$ consists of all $( (b,t), g(b,t))$ with $g\in \Gamma$ and $(b,t), g(b,t) \in B\times I$). Let $\theta_I$ be the probability measure on $I=[0,T]$ given by $d\theta_I=\frac{e^t}{e^T-1} dt$. So $\nu \times \theta_I$ is $\cR_I$-invariant.

{\bf Notational convention}. We change our notation and from now on we let $(X,\mu)$ denote an ergodic probability measure preserving action of $\Gamma$, and we consider the action of $\Gamma$ on $(B\times X,\nu\times \mu)$, which is an amenable action which is a class-bijective  extension of the amenable action of $\Gamma$ on $(B,\nu)$. 


The purpose of this section is to prove the following.
\begin{thm}\label{thm:maximalg}
Let $\cF=\{\cF_r\}_{r\in \II}$ be a Borel family of subset functions for $(B\times I, \nu \times \theta_I, \cR_I)$. Suppose $\cF$ is either regular or (asymptotically invariant, uniform and tempered). Let $\Gamma \cc (X,\mu)$ be a p.m.p. action. Let $\pi:B\times X \times I \to B\times I$ be the projection map $\pi(b,x,t)=(b,t)$ and let $\tcF=\{\tcF_r\}_{r\in \II}$ be the lift of $\cF$:
$$\tcF_r(x):=\pi^{-1}(\cF_r(\pi(b,x,t))) \cap [b,x,t] \quad \forall (b,x,t) \in B\times X \times I.$$

For $f \in L^1(B\times X \times I, \nu \times \mu \times \theta_I)$ and $(b,x,t) \in B\times X \times I$, define
\begin{eqnarray*}
\sM[f | \tcF](b,x,t) &:=& \sup_{r\in \II} \sA[|f| | \tcF_r](b,x,t)\\
\overline{\sM}[f | \tcF](b,x) &:=& \sup_{r\in \II} \frac{1}{T}\int_0^T \sA[|f| | \tcF_r](b,x,t)~dt\\
\sM[f | \tcF, \psi](x) &:=& \sup_{r\in \II} \frac{1}{T}\int \int_0^T \sA[|f| | \tcF_r](b,x,t)\psi(b)~dt d\nu(b).
\end{eqnarray*}
Then
\begin{enumerate}
\item  there exist constants $C_p$ for $p> 1$ such that for every $f\in L^p(B\times X\times I)$, 
$$\|\sM[f|\tcF]\|_p\le C_p \|f\|_p, \quad \|\overline{\sM}[f | \tcF]\|_p \le C_p \left(\frac{e^T-1}{T}\right)^{1/p}\|f\|_p.$$
Also if $\frac{1}{p}+\frac{1}{q} = 1$ and $p>1$ then $\| \sM[f | \tcF, \psi] \|_p \le C_p \left(\frac{e^T-1}{T}\right)^{1/p} \|\psi\|_q \|f\|_p$. 
\item There is also a constant $C_1>0$ such that if $f\in L\log L(B\times X \times I)$ then 
$$\|\sM[f|\tcF]\|_1\le C_1 \|f\|_{L\log L},\quad   \|\overline{\sM}[f | \tcF]\|_1 \le C_1 \left(\frac{e^T-1}{T}\right)\|f\|_{L\log L}.$$
If, in addition, $\psi \in L^\infty(B)$ then $\| \sM[f | \tcF, \psi] \|_1 \le C_1 \left(\frac{e^T-1}{T}\right) \|\psi\|_\infty \|f\|_{L\log L}$.
\end{enumerate}
The constants $C_p$, for $p\ge 1$, do not depend on $f$ or the action $\Gamma \cc (X,\mu)$.

\end{thm}


\begin{proof}
Let us first consider the case $p>1$ and $\frac{1}{p}+\frac{1}{q}=1$. By Theorem \ref{thm:maximalp}, there is a constant $C_p>0$ (independent of $f$ and the action $\Gamma \cc (X,\mu)$) such that $\|\sM[f|\tcF]\|_p \le C_p\|f\|_p$. Let us therefore turn  to the other two maximal operators. 

 By Jensen's inequality,
\begin{eqnarray*}
\| \overline{\sM}[f | \tcF] \|_p^p &=& \iint  \overline{\sM}[f | \tcF](b,x)^p ~d\nu(b)d\mu(x)\\
&=& \iint \left(\sup_{r\in \II} \frac{1}{T}\int_0^T \sA[|f| | \tcF_r](b,x,t)~dt\right)^p~d\nu(b)d\mu(x)\\
&\le& \iint \sup_{r\in \II} \frac{1}{T}\int_0^T \sA[|f| | \tcF_r](b,x,t)^p~dt ~d\nu(b)d\mu(x)\\
&\le& \iint  \frac{1}{T}\int_0^T\sup_{r\in \II} \sA[|f| | \tcF_r](b,x,t)^p~dt ~d\nu(b)d\mu(x)\\
&=& \iint  \frac{1}{T}\int_0^T  \sM[f | \tcF_r](b,x,t)^p~dt ~d\nu(b)d\mu(x)\\
&\le &  \frac{e^T-1}{T} \int_0^T \iint \sM[f | \tcF_r](b,x,t)^p ~d\nu(b)d\mu(x)d\theta_I(t)\\
&=&  \frac{e^T-1}{T} \| \sM[f | \tcF_r] \|_p^p \le C^p_p\left(\frac{e^T-1}{T}\right) \|f\|_p^p.
\end{eqnarray*}
 We conclude that 
$$\| \overline{\sM}[f | \tcF] \|_p \le  C_p\left(\frac{e^T-1}{T}\right)^{1/p} \|f\|_p.$$
Turning to the last maximal operator, by H\"older's inequality,
\begin{eqnarray*}
\| \sM[f | \tcF,\psi] \|_p^p &=& \iint \sM[f | \tcF,\psi](x)^p ~d\mu(x)\\
&=& \int \left(\sup_{r\in \II} \int \frac{1}{T}\int_0^T \sA[|f| | \tcF_r](b,x,t)\psi(b)~dtd\nu(b)\right)^p d\mu(x)\\
&\le& \int \sup_{r\in \II} \left( \frac{1}{T}\iint_0^T \sA[|f| | \tcF_r](b,x,t)^p~dt  d\nu(b)\right) \left(\iint_0^T \psi(b)^q ~dtd\nu(b)\right)^{p/q} d\mu(x)\\
&\le& \|\psi\|_q^p \int \sup_{r\in \II} \frac{1}{T}\iint_0^T \sA[|f| | \tcF_r](b,x,t)^p~dt  d\nu(b) d\mu(x)\\
&\le&\|\psi\|_q^p \frac{1}{T}\iiint_0^T \sM[f | \tcF](b,x,t)^p~dt  d\nu(b) d\mu(x)\\
&\le& \left(\frac{e^T-1}{T}\right)\|\psi\|_q^p \|\sM[f|\tcF]\|_p^p\le C_p^p\left(\frac{e^T-1}{T}\right)\|\psi\|_q^p \|f\|_p^p.
\end{eqnarray*}
So,
$$\| \sM[f | \tcF,\psi] \|_p \le C_p\left(\frac{e^T-1}{T}\right)^{1/p}\|\psi\|_q \|f\|_p.$$

As to the $L\log L$ results, let us now suppose $f \in L\log L(B\times X \times I)$ and $\psi \in L^\infty(B)$. By Theorem \ref{thm:maximalp}, there is a constant $C_1>0$ (independent of $f$ and the action $\Gamma \cc (X,\mu)$) such that $\|\sM[f|\tcF]\|_1 \le C_1\|f\|_{L\log L}$. The proof that $\| \overline{\sM}[f | \tcF] \|_1 \le  C_1\left(\frac{e^T-1}{T}\right) \|f\|_{L\log L}$ and 
$$\| \sM[f | \tcF,\psi] \|_1 \le C_1\left(\frac{e^T-1}{T}\right) \|\psi\|_\infty \|f\|_{L\log L}$$
 are similar to the proofs in the $p>1$ case.
\end{proof}

\section{General ergodic theorems from $III_1$ actions}\label{sec:general}


 Let $\Gamma \cc (B,\nu)$ be an action of a countable group on a standard probability space. We will assume the action is essentially free, amenable, weakly mixing and stable type $III_\lambda$ for some $\lambda>0$. From these assumptions and a choice of F\o lner sequence for a certain associated amenable equivalence relation, we will obtain in 
 \S 3.1 - \S 3.4 a family of pointwise ergodic sequences for $\Gamma$.
 
  Let us begin by explaining the terms mentioned above. {\em Essentially free} means that for a.e. $b \in B$ the stability group $\{g\in \Gamma:~gx=x\}$ is trivial. By {\em amenable action } we mean amenability in the sense of Zimmer \cite{Zi78}. {\em Weakly mixing} means that if $\Gamma \cc (X,\mu)$ is any ergodic p.m.p. (probability-measure-preserving) action then the product action $\Gamma \cc (B\times X,\nu\times \mu)$ is ergodic. We now turn to  define the (new) notion of stable type.

\subsection{The stable ratio set}\label{sec:AR}

Let $\Gamma$ be a countable group and $(B,\nu)$ a standard probability space on which $\Gamma$ acts by non-singular transformations. The {\em ratio set} of the Radon-Nikodym cocycle is a set $RS(\Gamma,B,\nu) \subset [0,+\infty]$ defined as follows: a finite number $r\in RS(\Gamma,B,\nu)$ if and only if for every positive measure set $A \subset B$ and $\epsilon>0$ there is a subset $A' \subset A$ of positive measure and an element $g\in \Gamma \setminus\{e\}$ such that 
\begin{enumerate}
\item $gA' \subset A$,
\item $| \frac{d\nu \circ g}{d\nu}(b)-r| < \epsilon$ for every $b \in A'$.
\end{enumerate}
The extended real number $+\infty \in RS(\Gamma,B,\nu)$ if and only if for every positive measure set $A \subset B$ and $n>0$ there is a subset $A' \subset A$ of positive measure and an element $g\in \Gamma \setminus\{e\}$ such that 
\begin{enumerate}
\item $gA' \subset A$,
\item $ \frac{d\nu \circ g}{d\nu}(b) > n$ for every $b \in A'$.
\end{enumerate}
The ratio set is also called the {\em asymptotic range} or {\em asymptotic ratio set}. By Proposition 8.5 of \cite{FM77}, if the action $\Gamma \cc (B,\nu)$ is ergodic then $RS(\Gamma,B,\nu)$ is a closed subset of $[0,\infty]$. Moreover, $RS(\Gamma,B,\nu) \setminus \{0,\infty\}$ is a multiplicative subgroup of $\RR_{>0}$. In the special case in which $\Gamma \cc (B,\nu)$ is an amenable action and a.e. orbit is infinite, it is known through work of W. Krieger \cite{Kr70} that there are four possibilities for $RS(\Gamma,B,\nu)$: either $RS(\Gamma,B,\nu) = \{1\}$ in which case the action is said to be type $II$; $RS(\Gamma,B,\nu) = \{0,1,+\infty\}$ which is called type $III_0$; $RS(\Gamma,B,\nu) =\{0,\lambda^n, +\infty: ~n\in \ZZ\}$ for some $\lambda \in (0,1)$ which is called type $III_\lambda$; $RS(\Gamma,B,\nu) = [0,+\infty]$ which is called type $III_1$. For a very readable review, see \cite{KW91}. There is also an extension to general cocycles taking values in an arbitrary locally compact group in section 8 of \cite{FM77}.

Observe that if $\Gamma \cc (X,\mu)$ is a p.m.p. action then the product action $\Gamma \cc (B\times X, \nu \times \mu)$ has ratio set $RS(\Gamma,B\times X, \nu \times \mu) \subset RS(\Gamma,B,\nu)$. Therefore, it makes sense to define the {\em stable} ratio set of $\Gamma \cc (B,\nu)$ by $SRS(\Gamma,B,\nu) = \cap RS(\Gamma, B\times X,\nu \times \mu)$ where the intersection is over all p.m.p. actions $G \cc (X,\mu)$. If $\Gamma \cc (B,\nu)$ is weakly mixing then $SRS(\Gamma,B,\nu)$ is a closed subset of $[0,\infty]$ and $SRS(\Gamma,B,\nu) \setminus \{0,\infty\}$ is a multiplicative subgroup of $\RR_{>0}$. 

In the case we are most interested in, $\Gamma \cc (B,\nu)$ is a weakly mixing amenable action in which a.e. orbit is infinite. Therefore $\Gamma \cc (B\times X,\nu \times \mu)$ is amenable and ergodic and there are only four possibilities for $SRS(\Gamma,B,\nu)$: either $SRS(\Gamma,B,\nu) = \{1\}$ in which case we say that $\Gamma \cc (B,\nu)$ is stable type $II$; $SRS(\Gamma,B,\nu) = \{0,1,+\infty\}$ which is called stable type $III_0$; $SRS(\Gamma,B,\nu) =\{0,\lambda^n, +\infty: ~n\in \ZZ\}$ for some $\lambda \in (0,1)$ which is stable type $III_\lambda$; $SRS(\Gamma,B,\nu) = [0,+\infty]$ which is called stable type $III_1$.

\subsection{The Maharam extension}\label{sec:maharam}
Suppose $\Gamma \cc (H,\eta)$ is a non-singular ergodic action on a standard probability space.  The group $\Gamma$ acts on $H \times \RR$ by 
\begin{equation}\label{maharam}  g(h,t):=\left(gh, t + \ln \left( \frac{d\eta\circ g^{-1}}{d\eta}(h) \right) \right).\end{equation}
Let $\theta$ be the measure on $\RR$ given by $d\theta(t)=e^t dt$.  The action above preserves the product measure $\eta \times \theta$. This construction is called the {\em Maharam extension} \cite{Ma64, Aa97}.

The group of real numbers acts on $H \times \RR$ by $\phi_t(h,t') := (h,t'+t)$ for $h\in H, t,t' \in \RR$. This action commutes with the action of $\Gamma$ and therefore descends to an $\RR$-action on the space of ergodic components of $\eta \times \theta $. This action is called the {\em Mackey range} of the Radon-Nikodym cocycle  \cite{Ma66}. It has also been called the {\em Poincar\'e flow} \cite{FM77} and the {\em Radon-Nikodym flow} \cite{Mo08}. 

\begin{lem}\label{lem:type}
Suppose $\Gamma \cc (H,\eta)$ is ergodic, amenable, essentially free and type $III_\lambda$ for some $0 < \lambda < 1$. Then there is a probability measure $\eta'$ on $H$ which is equivalent to $\eta$ such that for a.e. $h\in H$ and every $g\in \Gamma$, 
$$\frac{d\eta^\prime\circ g^{-1}}{d\eta^\prime}(h) \in \{\lambda^n:~n\in\ZZ\}.$$
\end{lem}

\begin{proof}
Because $\Gamma \cc (H,\eta)$ is ergodic, amenable and essentially free, this action is orbit equivalent to an action of $\ZZ$ (see \cite{CFW81}). Proposition 2.2 of \cite{KW91} now implies the result.
\end{proof}

\begin{lem}\label{lem:period0}
Suppose $\Gamma \cc (H,\eta)$ is ergodic, essentially free and type $III_\lambda$ for some $\lambda>0$. If $\lambda\ne 1$ then let $T=-\log(\lambda)$. If $\lambda =1$ then let $T>0$ be arbitrary. Then for every bounded Borel $\Gamma$-invariant function $f$ on $H \times \RR$, $f\circ \phi_T = f$ a.e. 
\end{lem}

\begin{proof}
This lemma follows from Proposition 8.3 and Theorem 8 of \cite{FM77}. To be precise, the cocycle $c$ appearing in \cite{FM77} is, for us, the logarithmic Radon-Nikodym cocycle on the $\Gamma$-orbit equivalence relation $\cR:=\{ (h,gh):~h\in H, g\in \Gamma\}$ on $H$. So $c:\cR \to \RR$, $c(h,h') = \log \frac{d\nu \circ g^{-1}}{d\nu}(h)$ where $g\in \Gamma$ is an element such that $gh=h'$. This element is unique for a.e. $h \in H$ because $\Gamma \cc (H,\eta)$ is essentially free. Then, the asymptotic range $r_*(c)$ is, by definition, $\log(RS(\Gamma,H,\eta) \cap (0,\infty))$ and the normalized proper range $npr(c)$ is the set of all positive real numbers $T$ such that for any $\Gamma$-invariant $f\in L^\infty(H\times \RR)$, $f(h,t)=f(h,t+T)$ for a.e. $(h,t)$ (by Proposition 8.3 of \cite{FM77}). By Theorem 8 of \cite{FM77}, $npr(c)=r_*(c)$.
\end{proof}

\begin{cor}\label{ergodicity}
Suppose $\Gamma \cc (H,\eta)$ is ergodic, essentially free and type $III_1$. Then $\Gamma \cc (H\times \RR, \eta \times \theta)$ is ergodic.
\end{cor}
\begin{proof}
Let $\cI$ be the sigma-algebra of Borel subsets of $H \times \RR$ that are invariant under the $\Gamma$ action and the flow $\{\phi_t\}_{t\in \RR}$. We claim that $\cI$ is trivial (i.e., every set $A \in \cI$ satisfies $\eta \times \theta(A)=0$ or $\eta\times \theta(A^c)=0$ where $A^c$ denotes the complement of $A$).  Indeed, if $A \in \cI$ then, since $A$ is invariant under the flow $\{\phi_t\}_{t\in \RR}$, $A=A_0\times \RR$ for some Borel set $A_0 \subset H$ (up to measure zero). Since $\Gamma \cc (H,\eta)$ is ergodic, $\eta(A_0)\in \{0,1\}$ which implies the claim. 

Lemma \ref{lem:period0} implies that any bounded Borel $\Gamma$-invariant function $f$ on $H\times \RR$ is invariant under the flow. By the claim above, this implies $f$ is constant a.e.. Therefore, $\Gamma \cc (H\times \RR, \eta \times \theta)$ is ergodic.

\end{proof}

\subsection{Random and non-random pointwise ergodic theorems}

Let $(B,\nu)$ be a standard probability space and $\{\zeta_r\}_{r \in \II}$ a family of maps $\zeta_r:B\times \Gamma \to [0,1]$ satisfying
$$\sum_{\gamma \in \Gamma} \zeta_r(b,\gamma)=1,\quad \textrm{ for a.e. } b\in B.$$
We say $\{\zeta_r\}_{r\in \II}$ is a {\em random pointwise ergodic family in $L^p$} if for every p.m.p. action $\Gamma \cc (X,\mu)$, every $f\in L^p(X,\mu)$ and a.e. $(b,x) \in B \times X$,
$$\lim_{r\to\infty} \sum_{\gamma \in \Gamma} \zeta_r(b,\gamma)f(\gamma x) = \EE[f|\cI](x)$$
where $\cI$ is the sigma-algebra of $\Gamma$-invariant Borel sets in $X$.


\begin{thm}\label{random-L-1}
Let $\Gamma \cc (B,\nu)$ be an action of a countable group on a standard probability space. We assume the action is essentially free, weakly mixing and stable type $III_1$. Let $\Gamma \cc (B\times \RR, \nu \times \theta)$ be the Maharam extension. Let $T>0$ be arbitrary,  $I=[0,T]$, and $\theta_I$ be the probability measure on $[0,T]$ given by $d\theta_I(t) = \frac{e^t}{e^T-1} dt$. Let $\cR_I$ be the equivalence relation on $B\times I$ given by restricting the orbit equivalence relation on $B\times \RR$ (so $\cR_I$ consists of all $( (b,t), g(b,t))$ with $g\in \Gamma$ and $(b,t), g(b,t) \in B\times I$). 

Let $\cF=\{\cF_r\}_{r\in \II}$ be a Borel family of subset functions for $(B\times I, \nu \times \theta_I, \cR_I)$. Suppose $\cF$ is either (asymptotically invariant and regular) or (asymptotically invariant, uniform and tempered). Define $\zeta_r: B \times I\times \Gamma \to [0,1]$ by
$$\zeta_r(b,t,\gamma):= \frac{1}{|\cF_r(b,t)|}1_{\cF_r(b,t)}(\gamma(b,t)).$$
Then $\{\zeta_r\}_{r\in \II}$ is a random pointwise ergodic family for $\Gamma$ in $L^1$.
\end{thm}

\begin{proof}
Let $\Gamma \cc (X,\mu)$ be a p.m.p. action and $f\in L^1(X) \subset L^1(B\times X \times I)$.  Then for any $(b,x,t)$, 
$$\sum_{\gamma \in \Gamma} \zeta_r(b,t,\gamma) f(\gamma x) = \sA[f|\tcF_r](b,x,t)\,,$$
where $\tcF$ is the list of the subset function $\cF$ from $B\times I$ to $B\times X\times I$. 

Without loss of generality we may assume $\Gamma \cc (X,\mu)$ is ergodic. Because $\Gamma \cc (B,\nu)$ is stable type $III_1$, Corollary \ref{ergodicity} implies that the equivalence relation $(B\times X \times I, \nu \times \mu \times \theta_I, \cR_I)$ is ergodic. By Theorem \ref{thm:pointwise2}, when $f\in L^1(X)$ 
\begin{eqnarray*}
\lim_{r\to\infty} \sA[f|\tcF](b,x,t) &=& \int_B \int_X \int_I f ~d\nu d\mu d\theta_I = \int_X f ~d\mu.
\end{eqnarray*}
This proves the result.
\end{proof}

We now turn to prove a (non-random) pointwise ergodic theorem for arbitrary ergodic p.m.p. actions of $\Gamma$, with respect to a fixed sequence of probability measures supported on $\Gamma$.  Here we establish convergence for functions in $L^p$, $p > 1$, as well as functions in $L\log L$, but not for all functions in $L^1$.

\begin{thm}\label{thm:general} 
Let $\Gamma \cc (B,\nu)$ be an action of a countable group on a standard probability space. We assume the action is essentially free, weakly mixing and stable type $III_1$.
 Let $\Gamma \cc (B\times \RR, \nu \times \theta)$ be the Maharam extension. Let $T>0$ be arbitrary, $I=[0,T]$, and $\theta_I$ be the probability measure on $[0,T]$ given by $d\theta_I(t) = \frac{e^t}{e^T-1} dt$. Let $\cR_I$ be the equivalence relation on $B\times I$ given by restricting the orbit equivalence relation on $B\times \RR$ (so $\cR_I$ consists of all $( (b,t), g(b,t))$ with $g\in \Gamma$ and $(b,t), g(b,t) \in B\times I$). 

Let $\cF=\{\cF_r\}_{r\in \II}$ be a Borel family of subset functions for $(B\times I, \nu \times \theta_I, \cR_I)$. Suppose $\cF$ is either (asymptotically invariant and regular) or (asymptotically invariant, uniform and tempered). 
 Define $\zeta_r: B\times \Gamma \to [0,1]$ defined by
$$\zeta_r(b,\gamma):= \frac{1}{T}  \int_0^T \frac{1}{|\cF_r(b,t)|}1_{\cF_r(b,t)}(\gamma(b,t))~dt.$$
Then $\{\zeta_r\}_{r\in \II}$ is a random pointwise ergodic family for $\Gamma$ in $L^p$ for every $p>1$ and in $L\log L$. 

If $\psi \in L^q(B)$ is a probability density function (so $\psi\ge 0$ and $\int \psi ~d\nu = 1$) and $\zeta^\psi_r:\Gamma \to [0,1]$ is defined by $\zeta^\psi_r(\gamma) = \int \zeta_r(b,\gamma)\psi(b)~d\nu(b)$ then $\{\zeta^\psi_r\}_{r\in \II}$ is a pointwise ergodic family in $L^p$ for every $p>1$ with $\frac{1}{p} + \frac{1}{q} \le 1$. If $\psi \in L^\infty$ then $\{\zeta^\psi_r\}_{r\in \II}$ is a pointwise ergodic family in $L \log L$. 
\end{thm}


\begin{proof}[Proof of Theorem \ref{thm:general}]
Without loss of generality, we may assume $\Gamma \cc (X,\mu)$ is ergodic. The Maharam extension of the product action $\Gamma \cc (B\times X,\nu\times \mu)$ is 
$$\Gamma \cc (B \times X \times \RR, \nu \times \mu \times \theta) \simeq \Gamma \cc (B \times \RR, \nu \times \theta) \times (X,\mu).$$

Suppose now that $f$ depends only on its $x$-argument (so $f(b,x,t)=f(x)$). Then for any $(b,x)$,
\begin{eqnarray*}
\sum_{\gamma \in \Gamma} \zeta_r(b,\gamma)f(\gamma x) &=&\sum_{\gamma \in \Gamma} \frac{1}{T}  \int_0^T \frac{1}{|\cF_r(b,t)|}1_{\cF_r(b,t)}(\gamma(b,t))f(\gamma x)~dt\\
&=& \frac{1}{T}  \int_0^T \sA[f|\tcF_r](b,x,t)~dt.
\end{eqnarray*}
Similarly, 
$$\sum_{\gamma \in \Gamma} \zeta^\psi_r(\gamma)f(\gamma x) =  \frac{1}{T}  \iint_0^T \sA[f|\tcF_r](b,x,t)\psi(b)~dtd\nu(b).$$
If $f\in L^\infty(X)$ then the bounded convergence theorem implies that for a.e. $(b,x)\in B\times X$,
\begin{eqnarray*}
\lim_{r\to\infty} \sum_{\gamma \in \Gamma} \zeta_r(b,\gamma) f(\gamma^{}x) &=& \lim_{r\to\infty} \frac{1}{T} \int_0^T \sA[f|\tcF_r](b,x,t)~dt\\
&=& \frac{1}{T} \int_0^T  \lim_{r\to\infty}\sA[f|\tcF_r](b,x,t)~dt\\
&=&  \frac{1}{T} \int_0^T   \EE[f|\cI(\widetilde{\cR}_I)](b,x,t)~dt = \int f ~d\mu(x).
\end{eqnarray*}
Above, $\widetilde{\cR}_I$ denotes the orbit-equivalence relation of the action $\Gamma \cc B\times X \times \RR$ restricted to $B \times X \times I$ and $\cI(\widetilde{\cR}_I)$ is the sigma-algebra of $\widetilde{\cR}_I$-invariant measurable sets. This proves $\{\zeta_r\}_{r\in \II}$ is a random pointwise ergodic sequence in $L^\infty$. Also for a.e. $x\in X$,
\begin{eqnarray*}
\lim_{r\to\infty} \sum_{\gamma \in \Gamma} \zeta^\psi_r(\gamma) f(\gamma^{}x) &=& \lim_{r\to\infty} \frac{1}{T} \iint_0^T \sA[f|\tcF_r](b,x,t)\psi(b)~dtd\nu(b)\\
&=&\frac{1}{T} \iint_0^T   \lim_{r\to\infty}  \sA[f|\tcF_r](b,x,t)\psi(b)~dtd\nu(b)\\
&=&  \frac{1}{T} \iint_0^T   \EE[f|\cI(\widetilde{\cR}_I)](b,x,t)\psi(b)~dtd\nu(b) = \int f ~d\mu(x).
\end{eqnarray*}
This proves $\{\zeta^\psi_r\}_{r\in \II}$ is a pointwise ergodic sequence in $L^\infty$.

Suppose now that $f \in L^p(X) \subset L^p(B\times X \times I)$ for some $p>1$. We will show that for a.e. $(b,x)$, $\lim_{r\to\infty} \sum_{\gamma \in \Gamma} \zeta_r(b,\gamma) f(\gamma^{}x) = \int f ~d\mu$. By replacing $f$ with $f - \int f~d\mu$ if necessary, we may assume $\int f ~d\mu=0$. 

Let $\epsilon>0$. Because $L^\infty(X)$ is dense in $L^p(X)$, there exists an element $f' \in L^\infty(X)$ such that $\|f - f'\|_p \le \epsilon$ and $\int f'~d\mu=0$. So for a.e. $(b,x) \in B\times X$,
\begin{eqnarray*}
\limsup_{r\to\infty} \left| \sum_{\gamma \in \Gamma} \zeta_r(b,\gamma) f(\gamma^{}x) \right|&\le & 
\limsup_{r\to\infty}  \left| \sum_{\gamma \in \Gamma} \zeta_r(b,\gamma) [f(\gamma^{}x) -f'(\gamma^{}x)]\right| +  \left|\lim_{r\to\infty} \sum_{\gamma \in \Gamma} \zeta_r(b,\gamma) f'(\gamma^{}x)\right| \\
&=&\limsup_{r\to\infty}  \left| \sum_{\gamma \in \Gamma} \zeta_r(b,\gamma) [f(\gamma^{}x) -f'(\gamma^{}x)]\right| \\
&=&\limsup_{r\to\infty}  \left| \frac{1}{T} \int_0^T \sA[f - f'|\cF_r](b,x,t)~dt\right| \\
&\le&  \overline{\sM}[f - f'|\cF](b,x).
\end{eqnarray*}
Thus if $F(b,x):=\limsup_{r\to\infty} \left| \sum_{\gamma \in \Gamma} \zeta_r(b,\gamma) f(\gamma^{}x) \right|$, then 
$$\|F\|_p \le \| \overline{\sM}[f - f'|\cF] \|_p \le C'_p \| f-f'\|_p \le C'_p \epsilon$$
for some constant $C'_p>0$ (that is independent of $f$ and $f'$) by Theorem \ref{thm:maximalg}. Since $\epsilon$ is arbitrary, $\|F\|_p = 0$ which implies 
$$\lim_{r\to\infty} \sum_{\gamma \in \Gamma} \zeta_r(b,\gamma) f(\gamma^{}x) =0$$
for a.e. $(b,x)$ as required. This proves $\{\zeta_r\}_{r\in \II}$ is a random pointwise ergodic sequence in $L^p$ for every $p>1$.

Now suppose $p>1$ and $\frac{1}{p}+\frac{1}{q}=1$. Let $f,f'$ be as above. Then for a.e. $x\in X$,
\begin{eqnarray*}
\limsup_{r\to\infty} \left| \sum_{\gamma \in \Gamma} \zeta^\psi_r(\gamma) f(\gamma^{}x) \right|&\le & 
\limsup_{r\to\infty}  \left| \sum_{\gamma \in \Gamma} \zeta^\psi_r(\gamma) [f(\gamma^{}x) -f'(\gamma^{}x)]\right| +  \left|\lim_{r\to\infty} \sum_{\gamma \in \Gamma} \zeta^\psi_r(\gamma) f'(\gamma^{}x)\right| \\
&=&\limsup_{r\to\infty}  \left| \sum_{\gamma \in \Gamma} \zeta^\psi_r(\gamma) [f(\gamma^{}x) -f'(\gamma^{}x)]\right| \\
&=&\limsup_{r\to\infty}  \left| \frac{1}{T} \iint_0^T \sA[f - f'|\cF_r](b,x,t)\psi(b)~dtd\nu(b)\right| \\
&\le&  \sM[f - f'|\cF, \psi](x).
\end{eqnarray*}
Thus if $F(x):=\limsup_{r\to\infty} \left| \sum_{\gamma \in \Gamma} \zeta^\psi_r(\gamma) f(\gamma^{}x) \right|$, then 
$$\|F\|_p \le \| \sM[f - f'|\cF,\psi] \|_p \le C''_p \| f-f'\|_p \le C''_p \epsilon$$
for some constant $C''_p>0$ (that is independent of $f$ and $f'$ but may depend on $\psi$) by Theorem \ref{thm:maximalg}. Since $\epsilon$ is arbitrary, $\|F\|_p = 0$ which implies 
$$\lim_{r\to\infty} \sum_{\gamma \in \Gamma} \zeta^\psi_r(\gamma) f(\gamma^{}x) =0$$
for a.e. $x$ as required. This proves $\{\zeta^\psi_r\}_{r\in \II}$ is a pointwise ergodic sequence in $L^p$ for every $p>1$.

The last two cases to handle occur when $f \in L\log L(X)$ and $\psi \in L^\infty(B)$. The proofs of these cases are similar to the proofs above.
\end{proof}

\begin{remark} {\bf The type $ II_1$ case.}
Theorem  \ref{thm:general}  applies, in particular, to any {\it amenable} group 
which admits a free weakly-mixing action of stable type $III_\lambda$ for $\lambda > 0$, for example 
 when a non-trivial Poisson boundary with these properties exists. 
However, when G is an amenable group, we can also use actions of type $II_1$ to produce 
pointwise ergodic sequences on $G$. Indeed, consider  a weakly mixing measure-preserving 
action on a probability space $(B,\nu)$. This action is of course amenable, and {\it any } (uniform tempered, or regular) 
 F\o lner sequence  
for the orbit equivalence relation of $B$ induces a random  
pointwise ergodic sequence in $L^1$ for the $G$-action on $X$. By averaging a probability  
distribution $\psi$ on $B$ we also obtain a pointwise ergodic sequence on $G$ for its action on $X$. The proof  
is straightforward using the arguments in the proof of Theorem   \ref{thm:general}. 

\end{remark}

\begin{remark}{\bf Convergence and identification of the limit.}  Let $B$ be any free amenable action of a countable group $\Gamma$, not necessarily 
of stable type $III_\lambda$ or weak mixing. Thus for any (uniform tempered, or regular) F\o lner sequence on the $\Gamma$-orbit equivalence relation restricted to $B \times [0,T]$ (for any $T>0$) we obtain a random sequence $\zeta_r$ of averages which converge pointwise almost surely and in $L^1$-norm, the limit being the conditional expectation on the $\sigma$-algebra of relation-invariant sets by Theorem \ref{thm:pointwise}. Averaging them further w.r.t. a probability density $\psi$ on $B \times [0,T]$ we obtain averaging sequences $\zeta_r^\psi$ on $\Gamma$ which converge pointwise almost surely. Thus amenability of $B$ suffices to obtain convergence almost surely, but may not be sufficient to identify the limit of $\zeta_r^\psi(f)$ as the ergodic mean. 
Our arguments establishing this fact in Theorem \ref{thm:general} depend crucially on weak-mixing and stable type $III_1$.
It is interesting to note that in \cite{BKK11} the authors prove  pointwise convergence of uniform averages of  spherical  measures on Markov groups, but they do not 
identify the limit function. 
\end{remark}

\subsection{Lattices and actions of stable type $III_1$}\label{type III}
Summarizing our results thus far, Theorem \ref{thm:general} provides the following recipe to prove pointwise ergodic theorems for an arbitrary group $\Gamma$. First, find an  essentially free, weakly mixing, amenable action $\Gamma \cc (B,\nu)$ of stable type $III_1$. Let $T>0$, and choose a F\o lner family on $(B\times [0,T], \nu_{[0,T]}, \cR(B \times [0,T]))$ which is uniform and tempered (or just regular). Such a family always exists by amenability, as noted in Proposition \ref{prop:tempered}. Finally choose a probability density $\psi$ on $B$. From these objects, a pointwise ergodic sequence is constructed. The maximal inequalities for the associated averages holds more generally (they do not depend on the stable type or the weak mixing hypothesis as shown in \S 3.3). 

There are several choices in this construction: the action $\Gamma \cc (B,\nu)$, the F\o lner family $\cF$, and the probability density $\psi$. It is an interesting problem to determine whether a given family of probability measure $\{\mu_r\}_{r>0}$ on $\Gamma$ arises from one of these constructions. For example, suppose $\Gamma$ acts cocompactly by isometries on a negatively curved manifold $(M,d)$ with a basepoint $x_0$ and $\beta_r$ is the uniform probability measure on $\{g\in \Gamma: d(gx_0,x_0)<r\}$. Then is $\beta_r$ a pointwise ergodic family? Does it arise from one of these constructions? In \cite{BN1} the authors used an explicit particular instance of this construction  to prove that spherical averages form a pointwise ergodic sequence for nonabelian free groups (up to a certain well-known periodicity phenomenon).

The importance of the action $\Gamma \cc (B,\nu)$ leads to the following question:
\begin{question}
Does every discrete group have an essentially free, weakly mixing, amenable action of stable type $III_\lambda$ for some $\lambda \in (0,1]$?
\end{question}
The requirement that the action be essentially free can be removed by the following device. Let $u$ be the uniform measure on $\{0,1\}$. $\Gamma$ acts on the product space $(\{0,1\}^\Gamma,u^\Gamma)$ by $g\cdot x(f)=x(g^{-1}f)~\forall x\in \{0,1\}^\Gamma, g,f\in \Gamma$. This is a {\em Bernoulli shift} action. If $\Gamma \cc (B,\nu)$ is any action then the product action $\Gamma \cc (B\times\{0,1\}^\Gamma,\nu \times u^\Gamma)$ is essentially free. Moreover, if $(B,\nu)$ has any one of the properties $\{$weakly mixing, amenable, stable type $III_\lambda\}$ then this product action has the same property.

The action of a group on any of its Poisson boundaries is amenable \cite{Zi78} and weakly mixing \cite{AL05} (indeed these actions are doubly ergodic with coefficients in Hilbert spaces by \cite{Ka03}). If $\Gamma$ is non-amenable, then these actions are necessarily of type $III_\lambda$ for some $\lambda \in [0,1]$. It may well be the case that the type of the action on a Poisson boundary is {\it never} $III_0$, but this problem is still open. 

We are unaware of any previous study of the {\em stable} type of an amenable action. However there are results on the types of boundary actions. For example,  in \cite{INO08} it is proven that the Poisson boundary of a random walk on a Gromov hyperbolic group induced by a nondegenerate measure on $\Gamma$ of finite support is never of type $III_0$. In \cite{Su78, Su82}, Sullivan proved that the recurrent part of an action of a discrete conformal group on the sphere $\mathbb{S}^d$ relative to the Lebesgue measure is type $III_1$. Spatzier \cite{Sp87} showed that if $\Gamma$ is the fundamental group of a compact connected negatively curved manifold then the action of $\Gamma$ on the sphere at infinity of the universal cover is also of $III_1$. The types of harmonic measures on free groups were computed by Ramagge and Robertson \cite{RR97} and Okayasu \cite{Ok03}. 

An important class of discrete groups for which the type of the boundary action is known is that of irreducible lattices  in  connected semisimple Lie groups with finite center and no compact factors. 
Let $G$ be such a group and $\Gamma\subset G$ an irreducible lattice subgroup. 
The maximal boundary $B=G/P$, where $P$ is a minimal parabolic subgroup, carries a unique $G$-quasi-invariant measure class, denoted $\nu$. As to the stable type, we have :

\begin{proposition}
The action of $\Gamma$ on $(G/P,\nu)$ is amenable, weak mixing and essentially free, and of stable type $III_1$. 
\end{proposition}
\begin{proof}
Recall the duality principle for ergodicity on homogeneous spaces  \cite{Mo66}: if $G$ is an lcsc group, and $H_1, H_2$ are two closed subgroups, then $H_1$ is ergodic on $G/H_2$ if and only if $H_2$ is ergodic on $G/H_1$, if and only if $G$ is ergodic on $G/H_1\times G/H_2$. The measure classes 
taken on $G/H_1$ and $G/H_2$ are the unique $G$-invariant ones, and on $G/H_1\times G/H_2$ we take their product.  
A further aspect of the duality principle for homogeneous spaces is that  $G/H_2$ is an amenable $H_1$-space if $H_2$ is an amenable subgroup \cite[Cor. 4.3.7]{Zi84}. 

 The fact that the action of $\Gamma$ on $G/P$ is amenable and ergodic therefore follows  from the fact that the minimal parabolic subgroup $P$ is amenable and ergodic on $G/\Gamma$. Here we take the $G$-quasi-invariant measure class $\nu$ on $G/P$. 
  Let $P=MAN$ be the Levi decomposition of $P$.  
 Then up to $\nu$-measure zero $G/P\times G/P\cong G/A$, and since $A$ is ergodic on $G/\Gamma$ by the Howe-Moore ergodicity theorem, $\Gamma$ is ergodic on $G/P\times G/P$, namely $\Gamma$ is doubly ergodic. Similarly, $\Gamma$ is doubly ergodic on the product with coefficients in Hilbert spaces and in particular, the action of $\Gamma$ on $G/P$ is weak mixing. It is well-known that the $\Gamma$-action is also essentially free.


We now show that the type of the action is $III_1$, and then that the stable type is also $III_1$. 
First, note that the Maharam extension of the $G$-action on $G/P$, namely the action on $G/P\times \RR$ given by $g(hP,t)=(ghP,t+\log r_\nu(g,hP))$ is a transitive $G$-action. 
Indeed, the stability group of $(P,1)$ is the kernel of the modular homomorphism $\delta : P\to \RR_+^*$, which we denote by $L$. Now $r_\nu(p,P)=\delta(p)$ and the modular homomorphism is clearly surjective, so the well-defined map $G/L\to G/P\times \RR$ given by $gL\mapsto (gP,\log  r_\nu(g,P))$ is a $G$-equivariant isomorphism.  In particular $G$ is ergodic on the Maharam extension $G/P\times \RR$, but then so is 
the restriction of the $G$-action to $\Gamma$ by \cite[Thm 5.4]{Zi77}. Hence the Mackey range of the Radon-Nikodym derivative cocycle of the $\Gamma$-action on $G/P$ is the action of $\RR$ on a point and type of the $\Gamma$-action on the boundary is $III_1$. 

Consider now the action of $\Gamma$ on $(G/P\times X,\nu\times \mu)$, where $(X,\mu)$ is an ergodic $\Gamma$-action. In general, for any cocycle $\beta : \Gamma\times Y \to H$ defined on a $\Gamma$-space $Y$, the Mackey range of the cocycle $\beta$ coincides with the Mackey range of the cocycle 
$\tilde{\beta}$, defined for the $G$-action on the induced space $\text{Ind}_{\Gamma}^G(Y)=G/\Gamma\times _\alpha Y$ by $\tilde{\beta}(g,u\Gamma,y)=\beta(\alpha(g,u\Gamma),y)$. 
Here $\alpha : G\times G/\Gamma\to \Gamma$ is a cocycle associated with a section $\tau :G/\Gamma\to G$ with $\tau(\Gamma)=e$, and the notation $\times_\alpha$ denotes that the action on the second component is via the cocycle $\alpha$, namely $g(u\Gamma,y)=(gu\Gamma,\alpha(g,u\Gamma)y)$.

For a $\Gamma$-space $X$ consider the $G$-action $\text{Ind}_{\Gamma}^G(G/P\times X)$ induced by the $\Gamma$-action on $G/P\times X$. 
Note that the induced action is equivariantly isomorphic  to the product $G$-action on $G/P$ and $\text{Ind}_{\Gamma}^G(X)$ :
$$G/\Gamma\times_\alpha (G/P\times X)=\text{Ind}_{\Gamma}^G (G/P\times X)\cong G/P\times  ( \text{Ind}_{\Gamma}^G X)=G/P\times (G/\Gamma\times_\alpha X)$$
 This follows from the well-known fact that  the action $G/\Gamma\times_\alpha G/P$ of $G$ induced by the $\Gamma$-action on $G/P$ is isomorphic to the product $G$-action on $G/\Gamma\times G/P$.

If $(X,\mu)$ is a {\it measure-preserving} probability $\Gamma$-space the Mackey range of the Radon-Nikodym cocycle $r_{\nu\times \mu}$ on $G/P\times X$ coincides with the Mackey range of the Radon-Nikodym cocycle of the $G$-action on the induced space $\text{Ind}_{\Gamma}^G(G/P\times X)$. Indeed the latter coincides with $\tilde{r}_{\nu\times \mu}$, since the extension $G/P\times (G/\Gamma\times_\alpha X)\to G/P$ is a measure-preserving extension. 

 To find the Mackey range of the Radon-Nikodym cocycle in question, consider the Maharam extension $G/P\times  ( \text{Ind}_{\Gamma}^G X)\times \RR$ of the product action. The Maharam extension is clearly $G$-isomorphic to the product $G$-action on $G/L\times 
 \text{Ind}_{\Gamma}^G X$, since the Maharam extension $G/P\times \RR$ of $G/P$ is the $G$-action on $G/L$, as noted above. Now $\text{Ind}_{\Gamma}^G X$ is an ergodic p.m.p. $G$-action, and its restriction to $L$ is still ergodic. Indeed, while the $G$-action on $G/\Gamma\times_\alpha X$ may be a reducible action, the unipotent radical $N$ of $P$ acts ergodically  in any ergodic $G$-space.  This follows from the Mautner phenomenon : if $G_1$ is simple and non-compact, then any $L^2$-function invariant under the unipotent radical $N_1$ of a minimal parabolic subgrup $P_1$ of $G_1$ is in fact $G_1$-invariant. Hence if $G=\prod_{i=1}^N G_i$ is a product of simple non-compact groups, 
 $N=\prod_{i=1}^N N_i$ is ergodic in any ergodic $G$ space, and hence so is the larger subgroup $L$.

 It follows that the action of $G$ on $G/L\times  \text{Ind}_{\Gamma}^G X$ is also ergodic. Thus $G$ is ergodic on the Maharam extension, and the Mackey range of the $G$-action is the $\RR$-action on a point.  
 By the foregoing arguments, this is also the Mackey range of the Radon-Nikodym cocycle of the action of $\Gamma$ on $G/P\times X$, and thus the stable type is $III_1$.

 \end{proof}



\section{General ergodic theorems from $III_\lambda$ actions}\label{sec:lambda}

The purpose of this section is to obtain general ergodic theorems as in the previous section but under a different set of hypotheses. The main difference is that we assume throughout that the Radon-Nikodym derivatives for the action $\Gamma \cc (B,\nu)$ take values in a discrete group. More precisely, we assume there is some $\lambda \in (0,1)$ such that if 
\begin{eqnarray}\label{eqn:R}
R_\lambda(g,b):=-\log_\lambda\left( \frac{d\nu\circ g^{-1}}{d\nu}(b) \right)
\end{eqnarray}
then $R_\lambda(g,b) \in \ZZ$ for every $g\in \Gamma$ and a.e. $b\in B$. Let $\Gamma \cc B \times \ZZ$ by
$$g(b,n) = (gb,n+R_\lambda(g,b)).$$
This action, called the {\em discrete Maharam extension} preserves the product measure $\nu\times\theta_{\lambda}$ where $\theta_{\lambda}(\{n\}) = \lambda^{-n}$. Given an integer $N\ge 0$, let $I=\{0,\ldots, N-1\}$ and $\cR_I$ be the equivalence relation on $B\times I$ given by restricting the orbit-equivalence relation of $\Gamma \cc B\times \ZZ$. Let $\theta_{\lambda,I}$ be the probability measure on $I$ given by $\theta_{\lambda,I}(\{n\}) = \frac{\lambda^{-n}}{1+\cdots +\lambda^{-N+1}}$. 

Suppose also that $\Gamma \cc (X,\mu)$ is a p.m.p. action. Let $\Gamma \cc B\times X \times \ZZ$ by
$$g(b,x,n) = (gb,gx,n+R_\lambda(g,b)).$$
This action preserves the product measure $\nu\times\mu\times \theta_{\lambda}$. Let $\widetilde{\cR}_I$ be the equivalence relation on $B\times X \times I$ obtained by restricting the orbit-equivalence relation of the action $\Gamma \cc B\times X \times \ZZ$.

Our first step is to prove some maximal inequalities.

\begin{thm}\label{thm:maximalg-discrete}
Let $\cF=\{\cF_r\}_{r\in \II}$ be a Borel family of subset functions for $(B\times I, \nu \times \theta_{\lambda, I}, \cR_I)$. Suppose $\cF$ is either regular or (asymptotically invariant, uniform and tempered). We assume $\Gamma \cc (B,\nu)$ is essentially free. Let $\pi:B\times X \times I \to B\times I$ be the projection map $\pi(b,x,t)=(b,t)$ and let $\tcF=\{\tcF_r\}_{r\in \II}$ be the lift of $\cF$:
$$\tcF_r(x):=\pi^{-1}(\cF_r(\pi(b,x,t))) \cap [b,x,t] \quad \forall (b,x,t) \in B\times X \times I$$
where $[b,x,t]$ denotes the $\widetilde{\cR}_I$-equivalence class of $(b,x,t)$. Let $\psi \in L^1(B,\nu)$ be a probability density (i.e., $\psi \ge 0$ and $\int \psi~d\nu =1$). For $f \in L^1(B\times X \times I, \nu \times \mu \times \theta_{\lambda, I})$ and $(b,x,t) \in B\times X \times I$, define
\begin{eqnarray*}
\sM[f | \tcF, \psi](x) &:=& \sup_{r\in \II} \int \frac{1}{N+1}\sum_{t=0}^N \sA[|f| | \tcF_r](b,x,t)\psi(b)~ d\nu(b).
\end{eqnarray*}
Then
\begin{enumerate}
\item  there exist constants $C_p$ for $p> 1$ such that for every $f\in L^p(B\times X\times I)$, if $\frac{1}{p}+\frac{1}{q} = 1$ and $p>1$ then $\| \sM[f | \tcF, \psi] \|_p \le C_p  \|\psi\|_q \|f\|_p$. 
\item There is also a constant $C_1>0$ such that if $f\in L\log L(B\times X \times I)$ and $\psi \in L^\infty(B)$ then $\| \sM[f | \tcF, \psi] \|_1 \le C_1 \|\psi\|_\infty \|f\|_{L\log L}$.
\end{enumerate}
The constants $C_p$, for $p\ge 1$, do not depend on $f$ or the action $\Gamma \cc (X,\mu)$ but they may depend on $p$ and $N$.

\end{thm}
\begin{proof}
The proof is analogous to the proof of Theorem \ref{thm:maximalg}. We leave the details to the reader.
\end{proof}

\subsection{Ergodic decomposition}
We let $\Gamma \cc (B,\nu)$, $R_\lambda(g,b)$, $\lambda>0$, etc. be as in the previous subsection. The main result of this section is Corollary \ref{cor:expectation} which provides a formula for a certain average of conditional expectation operators. We also obtain an explicit description of the ergodic decomposition for the Maharam-extension action $\Gamma \cc (B\times X \times \ZZ,\nu \times \mu \times \theta_{\lambda})$.


\begin{lem}\label{lem:cocycle-partition}
Let $(W,\omega)$ be a standard probability space, $\Gamma \cc (W,\omega)$ an ergodic action preserving the measure-class. Let $\alpha:\Gamma \times W \to \ZZ$ be a cocycle for the action. Assume that for a.e. $x\in W$ and $n\in \ZZ$ there is a $g\in \Gamma$ with $\alpha(g,x)=n$. Let $N>0$ be an integer and $\cR'=\{(x,y)\in W:~ \exists g\in \Gamma, gx=y, \alpha(g,x) \equiv 0 \mod N\}$. Then there exists a positive integer $k$ such that $k|N$ and a partition $\{H_i\}_{i=0}^{k-1}$ of $Z$ such that
\begin{enumerate}
\item each $H_i$ has positive measure, is $\cR'$-saturated,
\item $\omega|_{H_i}$ is $\cR'$-ergodic,
\item for a.e. $x\in H_i$, $\forall g \in \Gamma$, $gx \in H_j \Leftrightarrow \alpha(g,x) \equiv j-i \mod k$.
\end{enumerate}
\end{lem}

\begin{proof}
Let $\cR$ be the orbit-equivalence relation on $W$. That is, $\cR$ is the set of all $(x,gx)$ for $x\in W, g\in \Gamma$. The relation $\cR'$ is a sub-equivalence relation of $\cR$ of index $N$ (i.e., for a.e. $x\in W$ the $\cR$-equivalence class of $x$ contains $N$ distinct $\cR'$-equivalence classes). This is because for a.e. $x\in W$ and $n\in \ZZ$ there is a $g\in \Gamma$ with $\alpha(g,x)=n$.

Because $\Gamma \cc (W,\omega)$ is ergodic, any $\cR'$-invariant measurable function $f$ must take on at most $N$ different values (after ignoring a measure zero set). So the ergodic decomposition of $\omega$ with respect to $\cR'$ contains $k$ components for some $k \le N$. 
By the ergodic decomposition theorem, there exists a measurable partition $\{H_i\}_{i=0}^{k-1}$ of $W$ such that each $H_i$ has positive measure, each $H_i$ is $\cR'$-saturated and $\omega|_{H_i}$ is $\cR'$-ergodic. 


For $j\in \{0,\ldots, k-1\}$, let $F_j$ be the function on $W$ defined by: $F_j(x)$ is the set of all $n + N \ZZ \in \ZZ/N\ZZ$ such that there exists $g\in \Gamma$ with $gx \in H_j$ and $\alpha(g,x) \equiv n \mod N$. We claim that $F_j$ is $\cR'$-invariant a.e.. To see this, let $x \in W$, $n + N\ZZ \in F_j(x)$ and $g\in \Gamma$ with $gx \in H_j$ and $\alpha(g,x) \equiv n \mod N$. Let $y$ be $\cR'$-equivalent to $x$. So there exists $g_0\in \Gamma$ with $y=g_0x$ and $\alpha(g_0,x)\equiv 0 \mod N$. 
Thus 
$$\alpha( gg_0^{-1},y) = \alpha(gg_0^{-1},g_0x) = \alpha(g,x) + \alpha(g_0^{-1},g_0x) =  \alpha(g,x) - \alpha(g_0,x) \equiv \alpha(g,x) \mod N.$$
Because $gg_0^{-1}y=gx \in H_j$ this proves that $n + N \ZZ \in F_j(y)$. Since $x,y,n$ are arbitrary, this implies $F_j$ is $\cR'$-invariant. By ergodicity, $F_j$ is constant on each $H_i$.

Let $G$ be the function on $W$ defined by $G(x)=F_i(x)$ whenever $x\in H_i$. We claim that $G$ is $\Gamma$-invariant a.e.. So suppose $x \in H_i, y\in H_j$ and $y=g_0x$ for some $g_0\in \Gamma$. Let $n + N \ZZ \in G(x)$. By definition, there exists $g_1 \in \Gamma$ with $g_1 x \in H_i$ and $\alpha(g_1,x) \equiv n \mod N$. Because $F_j$ is constant on $H_i$, there exists $g_2\in \Gamma$ with $g_2(g_1x) \in H_j$ and $\alpha(g_2,g_1x) \equiv \alpha(g_0,x)$. Thus $g_2g_1g_0^{-1} y \in H_j$ and
\begin{eqnarray*}
\alpha(g_2g_1g_0^{-1},y) &=& \alpha(g_2g_1g_0^{-1},g_0x) = \alpha(g_2,g_1x) + \alpha(g_1g_0^{-1},g_0x)\\
&=& \alpha(g_2,g_1x) + \alpha(g_1,x) - \alpha(g_0,x) \equiv \alpha(g_1,x) \equiv n \mod N.
\end{eqnarray*}
Since $x,y,n$ are arbitrary, this shows that $G$ is $\Gamma$-invariant. By ergodicity, $G$ is constant a.e.. By the cocycle equation, there is a subgroup $G_0 < \ZZ/N\ZZ$ such that $G(x)=G_0$ for a.e. $x$.


Let $G(i,j) \subset \ZZ/N\ZZ$ be the subset satisfying: for a.e. $x\in H_i$, $F_j(x) = G(i,j)$. We claim that there is an integer $t(i,j)$ such that $t(i,j) + G_0 = G(i,j)$. Indeed, if $n + N\ZZ, m+N\ZZ \in G(i,j)$ then for a.e. $x\in H_i$ there exist $g_n, g_m \in \Gamma$ with $g_nx, g_mx \in H_j$ and $\alpha(g_n,x)\equiv n \mod N$, $\alpha(g_m,x)\equiv m \mod N$. Therefore, 
$$\alpha(g_mg_n^{-1}, g_nx)  = \alpha(g_m,x) - \alpha(g_n,x) \equiv m-n \mod N$$
and $g_mg_n^{-1}(g_nx) \in H_j$. This implies $m-n \in G_0$ which establishes the claim.

We claim that if $j_1 \ne j_2$ then $G(i,j_1) \cap G(i,j_2) = \emptyset$. Indeed, if $n + N \ZZ \in G(i,j_1) \cap G(i,j_2)$ then for a.e. $x\in H_i$ there exists $g_1,g_2 \in \Gamma$ such that $g_1x \in H_{j_1}, g_2x \in H_{j_2}$, $\alpha(g_1,x) \equiv \alpha(g_2,x) \equiv n \mod N$. Therefore, $g_2x \in H_{j_2}$, $g_1g_2^{-1}(g_2x) \in H_{j_1}$ and $\alpha(g_1g_2^{-1},g_2x) \equiv 0 \mod N$. This contradicts the fact that $H_{j_2}$ is $\cR'$-saturated. So the claim is proven.

For each $i$, $\cup_{j=0}^{k-1} G(i,j)$ partitions the group $\ZZ/N\ZZ$ into cosets of $G_0$ (and thus $k | N$ and $G_0$ is generated by $k+N\ZZ$). So after re-indexing the $H_i$'s if necessary, we may assume that $G(0,j) = j + G_0$ for each $j$. 

We claim that $G(i,j) = G(0, j-i)=j-i+G_0$ (indices mod $k$). Let $n + N\ZZ \in G(i,j)$. So for a.e. $x \in H_i$, there is a $g_0\in \Gamma$ with $g_0x \in H_j$ and $\alpha(g_0,x)  \equiv n  \mod N$. By ergodicity for a.e. such $x$ there exists $g_1\in \Gamma$ such that $g_1x \in H_0$. Then $g_0g_1^{-1}(g_1x) \in H_j$, so $\alpha(g_0g_1^{-1}, g_1x) + N \ZZ \in j + G_0$ and $\alpha(g_1^{-1},g_1x) + N \ZZ \in i +G_0$. But
$$\alpha(g_0g_1^{-1}, g_1x) = \alpha(g_0,x) + \alpha(g_1^{-1},g_1x).$$
So $n \equiv \alpha(g_0,x) \equiv j - i \mod k$. This proves the claim. Thus for a.e. $x\in H_i$, $\forall g \in \Gamma$, $gx \in H_j \Leftrightarrow \alpha(g,x) \equiv j-i \mod k$.


 \end{proof}
 
\begin{lem}\label{lem:period0b}
Suppose $\Gamma \cc (B,\nu)$ is an essentially free, weakly mixing, amenable action of type $III_\lambda$ and stable type $III_\tau$ for some $\lambda,\tau \in (0,1)$. Suppose as well that $\lambda^N=\tau$ for some integer $N\ge 1$ and $R_\lambda(g,b) \in \ZZ$ where $R_\lambda(\cdot,\cdot)$ is defined as in (\ref{eqn:R}). Let $\Gamma \cc (X,\mu)$ be an ergodic p.m.p. action. Let $\Gamma \cc B \times X \times \ZZ$ by 
$$g(b,x,n) = \left(gb,gx, n +R_\lambda(g,b) \right).$$
Then for every bounded Borel $\Gamma$-invariant function $f$ on $B\times X \times \RR$, $f(b,x,n) = f(b,x,n+N)$ for a.e $(b,x,n)$.
\end{lem}

\begin{proof}
This lemma follows from Proposition 8.3 and Theorem 8 of \cite{FM77}. To be precise, the cocycle $c$ appearing in \cite{FM77} is, for us, $R_\lambda$. So $c:\cR \to \ZZ$, $c((b,x),(b',x')) = R_\lambda(g,b)$ where $g\in \Gamma$ is an element such that $gb=b'$. This element is unique for a.e. $b \in B$ because $\Gamma \cc (B,\nu)$ is essentially free. Then, the asymptotic range $r_*(c)$ is, by definition, $\log_\lambda(RS(\Gamma, B \times X, \nu\times \mu) \cap (0,\infty))$. Because $\Gamma \cc (B,\nu)$ has stable type $III_\tau$ with $\tau=\lambda^N$, $RS(\Gamma, B \times X,\nu\times \mu) \supset \{ \lambda^{Ni}:~ i \in \ZZ\}$. So $N\ZZ \subset r_*(c)$.

The normalized proper range $npr(c)$ is the set of all positive integers $T$ such that for any $\Gamma$-invariant $f\in L^\infty(B\times X\times \ZZ)$, $f(b,x,n)=f(b,x,n+T)$ for a.e. $(b,x,n)$ (by Proposition 8.3 of \cite{FM77}). By Theorem 8 of \cite{FM77}, $npr(c)=r_*(c)$.
\end{proof}

\begin{lem}
Let the hypotheses be as in the previous lemma. Then there is a partition $\{H_i\}_{i=0}^{k-1}$ of $B\times X$ such that $k \mid N$ and
\begin{enumerate}
\item if $K_i = \bigcup_{g\in \Gamma} g(H_i \times \{0\})$ then $\{K_i\}_{i=0}^{k-1}$ partitions $B\times X \times \ZZ$ up to a set of measure zero;
\item $K_i = \bigcup_{j\in \ZZ} H_{i+j} \times \{j\}$ where the indices on $H$ are taken mod $k$;
\item for each $i$, $\Gamma \cc (K_i,\nu\times\mu\times\theta_{\lambda}|_{K_i})$ is ergodic;
\item $\nu \times \mu(H_i) = 1/k$ for all $i$.
\end{enumerate}
\end{lem}

\begin{proof}
Let $\Gamma \times \ZZ$ act on $B\times X \times \ZZ$ by 
$$(g,m)(b,x,n) =  \left(gb,gx, m+n + R_\lambda(g,h)\right).$$
We claim that this action is ergodic. Indeed, any $\Gamma \times \ZZ$-invariant Borel set $A$ is necessarily of the form $A = A_0 \times \ZZ$ for some $\Gamma$-invariant $A_0 \subset B\times X$. By ergodicity of the action $\Gamma \cc (B\times X,\mu\times \nu)$, $\mu\times\nu(A_0) \in \{0,1\}$ which implies $A$ or its complement has $\mu\times\nu \times \theta_{\lambda}$-measure zero, establishing the claim.

Let $f$ be a bounded $\Gamma$-invariant Borel function on $B\times X \times \ZZ$. By Lemma \ref{lem:period0b}, for a.e. $(b,x,n) \in B\times X\times \ZZ$, $f(b,x,n) = f(b,x,n+N)$. That is, $f$ is invariant under the action of the subgroup $\Gamma \times N\ZZ$. 

Let $\cR=\{((b,x),g(b,x)) \in B\times X\times B\times X:~(b,x) \in B\times X, g\in \Gamma\}$ be the orbit-equivalence relation for the action $\Gamma \cc (B\times X,\nu\times \mu)$. Let $\cR'$ be the set of all $((b,x),g(b,x)) \in \cR$ such that $R_\lambda(g,b) \equiv 0 \mod N$. Because $\Gamma \cc (B,\nu)$ is type $III_\lambda$ and $R_\lambda(g,b)$ takes values in the integers, it follows that for a.e. $b \in B$, $R_\lambda(\cdot, b)$ maps onto $\ZZ$. By Lemma \ref{lem:cocycle-partition}, there exists a measurable partition $\{H_i\}_{i=0}^{k-1}$ of $B\times X$ such that $k\mid N$, each $H_i$ has positive measure, each $H_i$ is $\cR'$-saturated, $\nu\times\mu|_{H_i}$ is $\cR'$-ergodic for each $i$ and 
for a.e. $(b,x)\in H_i$, $\forall g \in \Gamma$, $g(b,x) \in H_j \Leftrightarrow R_\lambda(g,b) \equiv j-i \mod k$.

Let $K_i$ be the $\Gamma$-orbit of $H_i\times \{0\} \subset B\times X\times \ZZ$. Because $\Gamma$-invariance automatically implies $\Gamma \times N \ZZ$-invariance, each $K_i$ is $\Gamma \times N\ZZ$-invariant. Because $\{H_i\}_{i=0}^{k-1}$ partitions $B\times X$, it follows that $\{K_i\}_{i=0}^{k-1}$ partitions $B\times X \times \ZZ$ (up to measure zero sets). Also $K_i = \bigcup_{j\in \ZZ} H_{i+j} \times \{j\}$ where the indices on $H$ are taken mod $k$ (because of the last statement in the previous lemma).


The restriction of $\nu\times\mu\times \theta_{\lambda}$ to $K_i$ is ergodic for the $\Gamma$-action. To see this, let $f$ be a  bounded $\Gamma$-invariant Borel function with support in $K_i$. As mentioned above, $f$ is $\Gamma \times N\ZZ$-invariant. Therefore, if $((b,x),g(b,x)) \in \cR'$ (i.e., $R_\lambda(g,b) \in N\ZZ$) then $f(b,x,0)=f(gb,gx, R_\lambda(g,b)) = f(gb,gx,0)$. So the map $(b,x) \mapsto f(b,x,0)$ is $\cR'$-invariant. Because $\nu\times\mu|_{H_i}$ is $\cR'$-ergodic, $f$ restricted to $H_i \times \{0\}$ is constant. Because $K_i$ is the $\Gamma$-orbit of $H_i\times \{0\}$, this implies that $f$ is constant on $K_i$. Because $f$ is arbitrary, this proves the claim:  $\Gamma \cc (K_i,\nu\times\mu\times\theta_{\lambda}|_{K_i})$ is ergodic.

 
 
Let $F:B \times X \to [0,1]$ be the function defined a.e. by $F(b,x)=\mu(A_{(b,x)})$ where $A_{(b,x)}$ is defined by 
$$\{b\} \times A_{(b,x)} = (\{b\} \times X) \cap H_i$$
where $i$ is such that $(b,x) \in H_i$. We claim that $F$ is $\Gamma$-invariant. Indeed, if $(b,x) \in H_i$ and $g\in \Gamma$ then $g(b,x) \in H_{i+R_\lambda(g,b)}$ (index $\mod k$). Thus $A_{g(b,x)} = gA_{(b,x)}$ which, because $\mu$ is $\Gamma$-invariant, implies the claim. By ergodicity of $\Gamma \cc B\times X$, there is a constant $c>0$ such that $F=c$ a.e. 

Let $\pi_B: B\times X \to B$ be the projection map. Fix $i$ with $0\le i \le k-1$. For $b\in \pi_B(H_i)$, choose an element $x_b \in X$ with $(b,x_b) \in H_i$. Then
\begin{eqnarray}\label{eqn:H_i}
\nu\times\mu(H_i) = \int_{\pi_B(H_i)} F(b,x_b) ~d\nu(b) = c \cdot \nu (\pi_B(H_i)).
\end{eqnarray}

We claim that $\pi_B(H_i)=B$ (up to sets of measure zero) for every $i$. To see this, let $\cR'_B$ be the equivalence relation on $B$ given by $(b,gb) \in \cR'_B$ if and only if $R_\lambda(g,b) \equiv 0 \mod N$. We claim that $\nu$ is $\cR'_B$-ergodic. By Lemma \ref{lem:cocycle-partition}  there is a measurable partition $\{C_i\}_{i=0}^{m-1}$ of $B$ into $\cR'_B$-saturated positive measure sets such that $\nu$ restricted to each $C_i$ is $\cR'_B$-ergodic (for some integer $m\mid N$). Moreover, for $b \in C_i$, $gb \in C_{i+R_\lambda(g,b)}$ with indices $\mod m$.

Because $\Gamma \cc (B,\nu)$ is type $III_\lambda$, there exists a positive measure subset $C'_0\subset C_0$ and an element $g\in \Gamma \setminus \{e\}$ such that $gC'_0 \subset C_0$ and  for every $b \in C'_0$, $R_\lambda(g,b) =1$. Thus $0 \equiv 1 \mod m$. So $m=1$ and $\nu$  is $\cR'_B$-ergodic as claimed.

Because $H_i$ is $\cR'$-saturated, it follows that each $\pi_B(H_i)$ is $\cR'_B$-saturated. By ergodicity and since $\nu\times\mu(H_i)>0$, this implies $\pi_B(H_i)=B$ up to a set of measure zero. By (\ref{eqn:H_i}), $\nu\times\mu(H_i) = c$ for every $i$. Since $\sum_{i=0}^{k-1} \nu\times\mu(H_i)=1$, this implies $c=1/k$ which finishes the lemma.


\end{proof}

\begin{cor}\label{cor:expectation}
Let the hypotheses be as in the previous lemma. Let $I=\{0,\ldots, N-1\}$ and let $\widetilde{\cR}_I$ be the restricted orbit-equivalence relation on $B\times X \times I$. Let $\bK_i = K_i \cap B\times X \times I$. Then $\nu\times\mu\times\theta_{\lambda}(\bK_i)=\nu\times\mu\times\theta_{\lambda}(\bK_j)$ for every $i,j$. Also let $\widetilde{\eta}_i$ be the restriction of $\nu\times\mu\times\theta_{\lambda}$ to $\bK_i$ and normalized to have total mass $1$. Then each $\widetilde{\eta_i}$ is $\widetilde{\cR}_I$-invariant, ergodic and
 $$\nu\times\mu\times\theta_{\lambda,I} = \frac{1}{k}\sum_{i=0}^{k-1} \widetilde{\eta}_i.$$
Thus for any $f \in L^1(B\times X\times I)$ and a.e. $(b,x)\in B \times X$,
$$\frac{1}{N} \sum_{i=0}^{N-1} \EE[f| \cI(\widetilde{\cR}_I)](b,x,i) = \int f~d(\nu\times\mu\times\theta_{\lambda,I})$$
where $ \EE[f| \cI(\widetilde{\cR}_I)]$ denotes the conditional expectation of $f$ on the sigma-algebra $\cI(\widetilde{\cR}_I)$ of $\widetilde{\cR}_I$-saturated Borel sets.
\end{cor}

\begin{proof}
By the previous lemma,
$$\nu\times\mu\times\theta_{\lambda}(\bK_i) = \sum_{j=0}^{N-1} \nu\times\mu(H_{i+j}) \lambda^{-j} = (1 + \lambda^{-1} + \cdots + \lambda^{-N+1})/k$$
with indices on $H$ taken $\mod k$. Because $\nu\times\mu\times\theta_{\lambda}|_{K_i}$ is ergodic for the action of $\Gamma$, it follows immediately that each $\widetilde{\eta_i}$ is $\widetilde{\cR}_I$-ergodic. Because $\nu\times\mu\times\theta_{\lambda}=\sum_{i=0}^{k-1} \nu\times\mu\times\theta_{\lambda}|_{K_i}$, it follows that 
$$\nu\times\mu\times\theta_{\lambda,I} =\frac{\nu\times\mu\times\theta_{\lambda}|_{B\times X \times I}}{1 + \lambda^{-1} + \cdots + \lambda^{-N+1}}=\frac{\sum_{i=0}^{k-1}  \nu\times\mu\times\theta_{\lambda}|_{\bK_i}}{1 + \lambda^{-1} + \cdots + \lambda^{-N+1}}= \frac{1}{k}\sum_{i=0}^{k-1} \widetilde{\eta}_i.$$

For any $f \in L^1(B\times X\times I)$ and a.e. $(b,x,n)$, $\EE[f| \cI(\widetilde{\cR}_I)](b,x,n) = \int f ~d\widetilde{\eta}_i$ a.e. where $i$ is such that $(b,x,n) \in \widetilde{K_i}$. This is well-defined because $\{\bK_i\}_{i=0}^{k-1}$ partitions $B\times X \times I$ (up to measure zero). By the previous lemma, if $(b,x,n) \in \bK_i$ then $(b,x,n+1) \in \bK_{i+1}$ (indices mod $N$ and $k$ respectively). Thus for a.e. $(b,x) \in B\times X$,
\begin{eqnarray*}
\frac{1}{N} \sum_{i=0}^{N-1} \EE[f| \cI(\widetilde{\cR}_I)](b,x,i) &=& \frac{1}{k} \sum_{i=0}^{k-1} \int f ~d\widetilde{\eta}_i =\int f~d(\nu\times\mu\times\theta_{\lambda,I}).
\end{eqnarray*}
\end{proof}

\subsection{Pointwise ergodic theorems from $III_\lambda$ actions}

\begin{thm}\label{thm:general-lambda} 
Let $\Gamma \cc (B,\nu)$ be an action of a countable group on a standard probability space. We assume the action is essentially free, weakly mixing, type $III_\lambda$ and stable type $III_\tau$ for some $\lambda, \tau \in (0,1)$ with $\tau=\lambda^N$ for some integer $N\ge 1$, 
and $R_\lambda(g,b) \in \ZZ$ where $R_\lambda(\cdot,\cdot)$ is defined as in (\ref{eqn:R}).
Let $I=\{0,\ldots,N-1\}$. Let $\cF=\{\cF_r\}_{r\in \II}$ be a Borel family of subset functions for $(B\times I, \nu \times \theta_{\lambda,I}, \cR_I)$. Suppose $\cF$ is either (asymptotically invariant and regular) or (asymptotically invariant, uniform and tempered). 
 Define $\zeta_r: B\times \Gamma \to [0,1]$ by
$$\zeta_r(b,\gamma):= \frac{1}{N}  \sum_{t=0}^{N-1} \frac{1}{|\cF_r(b,t)|}1_{\cF_r(b,t)}(\gamma(b,t)).$$
Then $\{\zeta_r\}_{r\in \II}$ is a random pointwise ergodic family for $\Gamma$ in $L^1$. 

If $\psi \in L^q(B)$ is a probability density function (so $\psi\ge 0$ and $\int \psi ~d\nu = 1$) and $\zeta^\psi_r:\Gamma \to [0,1]$ is defined by $\zeta^\psi_r(\gamma) = \int \zeta_r(b,\gamma)\psi(b)~d\nu(b)$ then $\{\zeta^\psi_r\}_{r\in \II}$ is a pointwise ergodic family in $L^p$ for every $p>1$ with $\frac{1}{p} + \frac{1}{q} \le 1$. If $\psi \in L^\infty(B)$ then $\{\zeta^\psi_r\}_{r\in \II}$ is a pointwise ergodic family in $L \log L$. 
\end{thm}


\begin{proof}[Proof of Theorem \ref{thm:general-lambda}]
Without loss of generality, we may assume $\Gamma \cc (X,\mu)$ is ergodic. Suppose now that $f \in L^1(B\times X \times I)$ depends only on its $x$-argument (so $f(b,x,t)=f(x)$). Recall that the averaging operator $\sA[f|\tcF_r]$ is defined by
$$\sA[f|\tcF_r](b,x,t):=|\tcF_r(b,x,t)|^{-1} \sum_{(b',x',t') \in \tcF(b,x,t)} f(b',x',t').$$
So for any $(b,x)$,
\begin{eqnarray*}
\sum_{\gamma \in \Gamma} \zeta_r(b,\gamma)f(\gamma x) &=&\sum_{\gamma \in \Gamma} \frac{1}{N}  \sum_{t=0}^{N-1} \frac{1}{|\cF_r(b,t)|}1_{\cF_r(b,t)}(\gamma(b,t))f(\gamma x)= \frac{1}{N}  \sum_{t=0}^{N-1} \sA[f|\tcF_r](b,x,t).
\end{eqnarray*}
Similarly, 
$$\sum_{\gamma \in \Gamma} \zeta^\psi_r(\gamma)f(\gamma x) =  \frac{1}{N}  \int \sum_{t=0}^{N-1} \sA[f|\tcF_r](b,x,t)\psi(b)~d\nu(b).$$
By Theorem \ref{thm:pointwise2} and Corollary \ref{cor:expectation} for a.e. $(b,x)$, 
\begin{eqnarray*}
\lim_{r\to\infty} \sum_{\gamma \in \Gamma} \zeta_r(b,\gamma) f(\gamma^{}x) &=& \lim_{r\to\infty}\frac{1}{N}  \sum_{t=0}^{N-1} \sA[f|\tcF_r](b,x,t)= \frac{1}{N}   \sum_{t=0}^{N-1}  \lim_{r\to\infty}\sA[f|\tcF_r](b,x,t)\\
&=&  \frac{1}{N}   \sum_{t=0}^{N-1}  \EE[f|\cI(\widetilde{\cR}_I)](b,x,t) = \int f ~d\mu(x).
\end{eqnarray*}
Above, $\widetilde{\cR}_I$ denotes the orbit-equivalence relation of the action $\Gamma \cc B\times X \times \ZZ$ restricted to $B \times X \times I$ and $\cI(\widetilde{\cR}_I)$ is the sigma-algebra of $\widetilde{\cR}_I$-invariant measurable sets. This proves $\{\zeta_r\}_{r\in \II}$ is a random pointwise ergodic sequence in $L^1$. If $f \in L^\infty$ then by the Bounded Convergence Theorem, for a.e. $x\in X$,
\begin{eqnarray*}
\lim_{r\to\infty} \sum_{\gamma \in \Gamma} \zeta^\psi_r(\gamma) f(\gamma^{}x) &=& \lim_{r\to\infty} \int \frac{1}{N}  \int \sum_{t=0}^{N-1} \sA[f|\tcF_r](b,x,t)\psi(b)~d\nu(b)\\
&=&\frac{1}{N}  \int \sum_{t=0}^{N-1}  \lim_{r\to\infty}  \sA[f|\tcF_r](b,x,t)\psi(b)~d\nu(b)\\
&=&\frac{1}{N}  \int \sum_{t=0}^{N-1}   \EE[f|\cI(\widetilde{\cR}_I)](b,x,t)\psi(b)~d\nu(b) = \int f ~d\mu(x).
\end{eqnarray*}
This proves $\{\zeta^\psi_r\}_{r\in \II}$ is a pointwise ergodic sequence in $L^\infty$.

Suppose now that $f \in L^p(B\times X \times I)$ for some $p>1$ and $\frac{1}{p}+\frac{1}{q}=1$. Without loss of generality, let us assume $\int f~d\mu=0$. Let $\epsilon>0$. Because $L^\infty(X)$ is dense in $L^p(X)$, there exists an element $f' \in L^\infty(X)$ such that $\|f - f'\|_p \le \epsilon$ and $\int f'~d\mu=0$.  Then for a.e. $x\in X$,
\begin{eqnarray*}
\limsup_{r\to\infty} \left| \sum_{\gamma \in \Gamma} \zeta^\psi_r(\gamma) f(\gamma^{}x) \right|&\le & 
\limsup_{r\to\infty}  \left| \sum_{\gamma \in \Gamma} \zeta^\psi_r(\gamma) [f(\gamma^{}x) -f'(\gamma^{}x)]\right| +  \left|\lim_{r\to\infty} \sum_{\gamma \in \Gamma} \zeta^\psi_r(\gamma) f'(\gamma^{}x)\right| \\
&=&\limsup_{r\to\infty}  \left| \sum_{\gamma \in \Gamma} \zeta^\psi_r(\gamma) [f(\gamma^{}x) -f'(\gamma^{}x)]\right| \\
&=&\limsup_{r\to\infty}  \left| \frac{1}{N}  \int \sum_{t=0}^{N-1} \sA[f - f'|\tcF_r](b,x,t)\psi(b)~d\nu(b)\right| \\
&\le&  \sM[f - f'|\tcF, \psi](x)
\end{eqnarray*}
where  $\sM[f - f'|\tcF, \psi](x)$ is as defined in Theorem \ref{thm:maximalg-discrete}. Thus if $F(x):=\limsup_{r\to\infty} \left| \sum_{\gamma \in \Gamma} \zeta^\psi_r(\gamma) f(\gamma^{}x) \right|$, then 
$$\|F\|_p \le \| \sM[f - f'|\tcF,\psi] \|_p \le C''_p \| f-f'\|_p \le C''_p \epsilon$$
for some constant $C''_p>0$ (that is independent of $f$ and $f'$ but may depend on $\psi$) by Theorem \ref{thm:maximalg-discrete}. Since $\epsilon$ is arbitrary, $\|F\|_p = 0$ which implies 
$$\lim_{r\to\infty} \sum_{\gamma \in \Gamma} \zeta^\psi_r(\gamma) f(\gamma^{}x) =0$$
for a.e. $x$ as required. This proves $\{\zeta^\psi_r\}_{r\in \II}$ is a pointwise ergodic sequence in $L^p$ for every $p>1$.

The last two cases to handle occur when $f \in L\log L(X)$ and $\psi \in L^\infty(B)$. The proofs of these cases are similar to the proofs above.
\end{proof}

\end{document}